\documentclass[12pt]{amsart}
\usepackage{amssymb,latexsym,amsmath,amscd,amsthm,amsfonts, enumerate}
\usepackage{multirow}
\usepackage{color}
\usepackage[all]{xy}
\usepackage{caption}
\usepackage{soul}
\usepackage{float}
\usepackage{titletoc}
\usepackage[normalem]{ulem}
\usepackage{graphicx}
\usepackage{parskip}

\usepackage{tikz,tikz-cd}
\usepackage{mathrsfs}
\usepackage[all]{xy}
\usepackage{tikz}
\usepackage{extarrows}
\usepackage{tikz-cd}
\usetikzlibrary{calc}
\usetikzlibrary{matrix,arrows,decorations.pathmorphing}
\usetikzlibrary{shapes.geometric,positioning}
\usetikzlibrary{arrows,decorations.pathmorphing,decorations.pathreplacing}
\usetikzlibrary{positioning,shapes,shadows,arrows}

\usepackage[colorlinks=true,pagebackref,hyperindex]{hyperref}
\newcommand\myshade{85}
\colorlet{mylinkcolor}{violet}
\colorlet{mycitecolor}{red}
\colorlet{myurlcolor}{cyan}

\hypersetup{
  linkcolor  = mylinkcolor!\myshade!black,
  citecolor  = mycitecolor!\myshade!black,
  urlcolor   = myurlcolor!\myshade!black,
  colorlinks = true,
}

\numberwithin{equation}{section}

\newtheorem{theorem}{Theorem}[section]
\newtheorem{theoremintro}{Theorem}
\newtheorem{corollaryintro}{Corollary}
\newtheorem{proposition}[theorem]{Proposition}
\newtheorem{proposition-definition}[theorem]{Proposition-Definition}

\newtheorem{corollary}[theorem]{Corollary}

\newtheorem{lemma}[theorem]{Lemma}

\theoremstyle{definition}
\newtheorem{remark}[theorem]{Remark}
\newtheorem{remarkintro}{Remark}
\newtheorem{example}[theorem]{Example}
\newtheorem{definition}[theorem]{Definition}

\newtheorem{theoremprime}{Theorem}
\makeatletter
\newenvironment{theorem'}%
  {%
   \begin{theoremprime}}
  {\end{theoremprime}}
\makeatother

\newcommand{\thick}{\mathsf{thick}}
\newcommand{\Hom}{\mathrm{Hom}}

\newcommand{\za}{\alpha}
\newcommand{\zb}{\beta}
\newcommand{\zd}{\delta}
\newcommand{\zD}{\Delta}

\newcommand{\Z}{\mathbb{Z}}

\newcommand{\aaa}{{\bf{a}}}
\newcommand{\bbb}{{\bf{b}}}
\newcommand{\ccc}{{\bf{c}}}

\newcommand{\xxx}{{\bf{x}}}
\newcommand{\yyy}{{\bf{y}}}

\newcommand{\mS}{{\mathbb{S}}}

\newcommand{\cals}{\Sigma}
\newcommand{\calt}{\mathcal{T}}

\newcommand{\calm}{M}

\newcommand{\calw}{\mathcal{W}}
\newcommand{\calk}{\mathcal{K}}
\newcommand{\bbp}{\mathbb{P}}
\newcommand{\dba}{\mathrm{per}(A)}
\newcommand{\Db}{\mathrm{per}}
\newcommand{\SSS}{\mathbb{S}}
\newcommand{\bul}{\bullet}

\newcommand{\uSdim}{\overline{{\rm \mathbb{S}dim}}\hspace{0.5mm}}
\newcommand{\lSdim}{\underline{{\rm \mathbb{S}dim}}\hspace{0.5mm}}
\newcommand{\gldim}{{\rm gldim}\hspace{0.5mm}}

\renewcommand{\P}{P^\bullet}

\definecolor{dark-green}{RGB}{14,150,2}
\definecolor{red}{RGB}{250,0,0}
\newcommand{\gpoint}{\color{black}{\circ}}
\newcommand{\rpoint}{\color{red}{\bullet}}

\begin{document}

\title[Entropy of the Serre functor]{Entropy of the Serre functor for partially wrapped Fukaya categories of surfaces with stops}

\author{Wen Chang}
\address{School of Mathematics and Statistics, Shaanxi Normal University, Xi'an 710062, China}
\email{changwen161@163.com}

\author{Alexey Elagin}
\address{School of Mathematical and Physical Sciences, University of Sheffield, 
 The Hicks Building,
 Sheffield 
 S3 7RH, UK}
\email{alexey.elagin@gmail.com}

\author{Sibylle Schroll}
\address{Insitut f\"ur Mathematik, Universit\"at zu K\"oln, Weyertal 86-90, K\"oln, Germany }
\email{schroll@math.uni-koeln.de}


\dedicatory{We dedicate this paper to Idun Reiten, whose work transformed the representation theory of algebras }

\date{\today}


\begin{abstract} We prove that the categorical entropy of the Serre functor $\mathbb{S}$ 
in the partially wrapped Fukaya category of a graded surface with stops is given by the function
$$H_t(\mathbb{S}) = \left\{\begin{array}{ll}
						(1-\min \Omega)t, & t\geq 0; \\
						(1-\max \Omega)t, & t\leq 0,
               \end{array}\right.$$	
where $\Omega = \{\frac{\omega_1}{m_1} \ldots, \frac{\omega_b}{m_b},0\}$,
 $\omega_i$ is the winding number of the $i$-th boundary component $\partial_i$ of the surface, and $m_i$ is the number of stops on $\partial_i$. 
It then follows 
that the upper and lower Serre dimensions are given by $1-\min \Omega$ and $1-\max \Omega$, respectively. 
Such Fukaya categories are equivalent to perfect derived categories of homologically smooth graded gentle algebras, so our results provide the entropy of the Serre functor and the Serre dimension for  these perfect derived categories.
Furthermore, for ungraded homologically smooth gentle algebras, we prove a Gromov--Yomdin-like equality between the entropy of the Serre functor and the logarithm of the spectral radius of the Coxeter transformation.
\end{abstract}

\maketitle
\setcounter{tocdepth}{2}

\tableofcontents

\section*{Introduction}\label{Introductions}

The study of \emph{entropy} in the context of triangulated categories, introduced in \cite{DHKK14}, provides a framework for quantifying the complexity of endofunctors. 
Inspired by topological entropy in dynamical systems, the categorical entropy $H_t(F)$ (see Definition~\ref{definition:entropy}) for a triangulated category $\mathcal{T}$ and an exact endofunctor $F: \mathcal{T} \to \mathcal{T}$ determines the growth rate of extensions between objects under iteration of $F$.

An important endofunctor in a triangulated category is the  \textit{Serre functor}, which exists in many triangulated categories of geometric and algebraic origin, such as derived categories of coherent sheaves and perfect derived categories of finite dimensional algebras that are Gorenstein. The computation of the entropy for Serre functors as well as that of other related endofunctors was initiated already in \cite{DHKK14}, where the entropy of the Serre functor of fractionally Calabi--Yau categories, the bounded derived categories of path algebras of acyclic quivers and the bounded derived categories of coherent sheaves of smooth projective varieties was calculated.

Recently, many further calculations of the entropy of the Serre functor and related endofunctors have been given, for example, see \cite{BK23a, E22, EL21, H24}.

In this paper, we compute the entropy of the Serre functor $\mathbb{S}$ of the partially wrapped Fukaya category $\mathcal{W}(\Sigma, M, \eta)$, where  $\Sigma$ is a smooth real compact oriented surface with boundary $\partial$, $M$ is a non-empty set of stops on $\partial$,  and $\eta$ is a line field on $\Sigma$. Our results show that the entropy of $\mathbb{S}$ is determined by  the number of stops in the boundary components and by their winding numbers.

For the statement of our main result, we fix the following notation (see Definition~\ref{definition:marked surface}). Let $\Sigma$ be an oriented surface with boundary components $\partial_1, \ldots, \partial_b$ and a finite set $M\subset \Sigma$ of stops   such that each boundary component $\partial_i$  has $m_i\geq 1$ stops.  Let $\eta$ be a line field on $\Sigma$. Denote by  $\mathcal{W}(\Sigma,M,\eta)$ the partially wrapped Fukaya category of $(\Sigma,M,\eta)$. Let~$\omega_i$ be the winding number of the boundary component $\partial_i$.

\begin{theoremintro}[Theorem~\ref{thm:main1gentle}]\label{thm:main1}
Let $\mathbb{S}$ be the Serre functor of the partially wrapped Fukaya category $\mathcal{W}(\Sigma, M, \eta)$. Suppose that $\Sigma$ is not a disc and that there are no stops in the interior of $\Sigma$. Set $\Omega = \{\frac{\omega_1}{m_1}, \ldots, \frac{\omega_b}{m_b},0\}$. Then the entropy of $\mathbb{S}$ is given  by the function from $\mathbb{R}$ to $\mathbb{R}$ sending $t \in \mathbb{R}$ to 
$$H_t(\mathbb{S}) = \left\{\begin{array}{ll}
						(1- \min  \Omega)t, & t\geq 0; \\ & \\
						(1-\max \Omega)t, & t\leq 0.
                        \end{array}\right.$$	
\end{theoremintro}

Partially wrapped Fukaya categories of marked graded surfaces are equivalent to the perfect derived categories of graded gentle algebras, and vice versa.
For a triple $(\Sigma, M,\eta)$ as above the authors of 
\cite{HKK17} 
construct collections of formal generators in $\mathcal{W}(\Sigma, M,\eta)$ whose endomorphism algebras are graded gentle algebras $A$, such that the perfect derived category $\dba$ of $A$ is equivalent to $\mathcal{W}(\Sigma, M,\eta)$. Conversely, for any homologically smooth graded gentle algebra $A$, in \cite{LP20} a triple $(\Sigma,M,\eta)$ is constructed such that $\mathcal{W}(\Sigma,M,\eta)$ is equivalent to $\dba$.
In this paper, we will mostly use ``algebraic'' terminology. In particular, we reformulate Theorem~\ref{thm:main1} in ``algebraic'' terms. 
Recall the AG-invariant~\cite{AG08, LP20} of $\dba$: it is the collection of pairs $\{(m_i, n_i)\}_{i=1,\ldots,b}$, where $n_i=m_i-\omega_i$. We set $\mathcal N=\{\frac{n_1}{m_1}, \ldots, \frac{n_b}{m_b},1\}$.

\begin{theorem'}
\label{thm:main1prime}
\emph{Let
$A$ be a homologically smooth graded gentle algebra, and $\mathbb{S}$ be the Serre functor on $\dba$. 
Suppose that $A$ is not derived equivalent to a hereditary algebra of type~$\mathbb A$. Then the entropy of $\mathbb{S}$ is given  by 
$$H_t(\mathbb{S}) = \left\{\begin{array}{ll}
						(\max  \mathcal N)t, & t\geq 0; \\ & \\
						(\min \mathcal N)t, & t\leq 0.
                        \end{array}\right.$$}
\end{theorem'}
In the case of a disc, the entropy of the partially wrapped Fukaya category is given in \cite{DHKK14}. Namely, in this case $\mathcal{W}(\Sigma, M, \eta)$ is triangle equivalent to the perfect derived category of a hereditary algebra of type $A_n$ for some $n$. In particular, it is fractionally Calabi--Yau of dimension $\frac{n-1}{n+1}$ and it is shown in  \cite{DHKK14} that the entropy  is given by $H_t(\mathbb{S})= \frac{n-1}{n+1}t=(1-\frac{2}{n+1})t$.  Note that in this case $\Sigma$ has $n+1$ stops in the boundary and the boundary has winding number $2$.

In the case that $\Sigma$ is an annulus with $m_1$ and $m_2$ marked points on the two boundary components $\partial_1$ and $\partial_2$ respectively,  we have that $\omega_1 = - \omega_2$. Suppose that $\omega = \omega_1 \geq 0$ then    
$$H_t(\mathbb{S}) = \left\{\begin{array}{ll}
(1 + \frac{\omega}{m_2})t, & t\geq 0; \\ & \\
(1- \frac{\omega}{m_1})t, & t\leq 0.
\end{array}\right.$$	
In particular, if  $\omega = 0$ then $\mathcal{W}(\Sigma, M, \eta)$ is triangle equivalent to the perfect derived category of a hereditary algebra of affine type $\Tilde{A}_{n}$ and we get $H_t(\mathbb{S}) = t$ which coincides with Theorem 2.16 in \cite{DHKK14}.

It is shown in \cite{DHKK14}  that,   for saturated triangulated categories, the entropy $H_t(F)$ of an exact endofunctor $F$ can be calculated in terms of dimensions of  Hom spaces and is given by the function $h_t(F) $ in equation (\ref{formula:hom}). In fact, the function $h_t(F)$ is defined for arbitrary (not necessarily saturated) Hom-finite triangulated category $\calt$. However, it is not clear whether in this generality $h_t(F)$ is equal to the categorical entropy $H_t(F)$.  

In this paper, in Theorem~\ref{thm:main1gentle}, we compute $h_t(F)$ for $F$ the Serre functor~$\SSS$ in 
the partially wrapped Fukaya category $\mathcal{W}(\Sigma, M, \eta)$, including those where  $\Sigma$ has fully stopped boundary components which we will refer to as stops in the interior of $\Sigma$. That is, we consider 
the perfect derived category $\dba$ of a (not necessarily homologically smooth) graded gentle algebra $A$.  
We note then that in the homologically smooth case, $h_t(\SSS)$ gives the categorical entropy $H_t(\SSS)$ of the Serre functor to deduce Theorem~\ref{thm:main1}.

\medskip
We now briefly discuss some consequences and relations of the entropy of the Serre functor to other notions of dimension of triangulated categories, such as the upper and lower Serre dimension \cite{EL21}, the Rouquier dimension \cite{R08}, and the global dimension of a triangulated category \cite{IQ23}. 

Let $\calt$ be an Ext-finite triangulated category with a Serre functor $\mathbb{S}$ and let $G$ and $G'$ be generators of $\calt$. The \emph{upper Serre dimension} and the \emph{lower Serre dimension}  of $\calt$ is defined  in \cite{EL21} as follows
$$\uSdim \calt=\limsup_{n\to +\infty} \frac{-e_-(G,\mS^n(G'))}n  \;\;\; \mbox{ and } \;\;\; \lSdim \calt=\liminf_{n\to +\infty} \frac{-e_+(G,\mS^n(G'))}n$$
 where we set \sloppy
$e_-(G,\mS^n(G'))=\inf\{i\mid \Hom^i(G,\mS^n(G'))\ne 0\}$ and where \sloppy
$e_+(G,\mS^n(G'))=\sup\{i\mid \Hom^i(G,\mS^n(G'))\ne 0\}.$
It is proven in \cite{EL21} that the Serre dimensions are well-defined, that is, they do not depend on the choice of generators.

\begin{theoremintro}
[Theorem~\ref{thm:Main2Section2}]\label{thm:main2} 
Let $\calt=\mathcal{W}(\Sigma, M, \eta)$ be the partially wrapped Fukaya category of $(\Sigma, M,\eta)$ (which now can have stops in the interior of $\Sigma$) and let  $\mathbb{S}$ be the Serre functor of $\calt$. Suppose further that  $\Sigma$ is not a disc with $\leq 1$ stops in the interior.
Then    
$$\uSdim \calt = (1-\min \Omega) \;\; \mbox{ and }\;\;  \lSdim \calt=(1-\max \Omega). $$
\end{theoremintro}

\begin{theorem'}
\label{thm:main2prime}
\emph{Let $A$ be a graded gentle algebra and $\SSS$ be the Serre functor on $\calt:=\dba$. 
Suppose that $A$ is not derived equivalent to a graded radical square zero Nakayama algebra. Then  
$$\uSdim \calt = \max \mathcal N \;\; \mbox{ and }\;\;  \lSdim \calt=\min \mathcal N. $$}
\end{theorem'}

\begin{remarkintro}
    Cases excluded in Theorems~\ref{thm:main2} and~\ref{thm:main2prime} correspond to each other and are treated in Section~\ref{section_examples}. Specifically, the geometrical model of the path algebra $kA_n$ of the linear quiver $A_n$ without relations (or with all quadratic relations) is a disc with $n+1$ stops on the boundary and no stops in the interior. Hence algebras with such geometrical models are derived equivalent to graded radical square zero \textbf{linear} Nakayama algebras. These algebras are considered in Example~\ref{example_1}. 
    
    Graded radical square zero \textbf{cyclic} Nakayama algebras are studied in Example~\ref{example_11}. Their geometric models are discs with one stop in the interior. Such marked graded surfaces are described by two parameters: the number of stops on the boundary and the winding number around the boundary. Any pair of parameters is realized by some graded radical square zero cyclic Nakayama algebra, hence any algebra with geometric model of this type is derived equivalent to one of the algebras from Example~\ref{example_11}.    

    In both cases, the results are given by 
    $$\uSdim \calt =  \lSdim \calt=1- \omega_1/m_1=n_1/m_1,$$
    where $m_1,n_1,\omega_1$ are associated with the only boundary component. So, in general the formulas from Theorems~\ref{thm:main2} and~\ref{thm:main2prime} are not valid.
\end{remarkintro}

The upper and lower Serre dimensions of several related algebras have been computed in \cite{E22}. This includes,  for example, the cases of the graded $n$-Kronecker quiver \cite[Theorem 1.3]{E22}, as well as some gentle algebras corresponding to an annulus with two or three marked points \cite[Section 9]{E22}. 

The Rouquier dimension of $\dba$ for gentle algebras (homologically smooth, non-graded) has been calculated in \cite{BD17}: it is $0$ if the algebra $A$ is derived equivalent to a hereditary algebra of type $\mathbb A$, and $1$ otherwise.

It follows from Theorem~\ref{thm:main2} that for partially wrapped Fukaya categories of surfaces with stops, Conjecture 1.2 in \cite{EL21} holds, see also (5) in~\cite{E22}. That is, we have the following

\begin{corollaryintro}\label{cor:intro}
    Let $\mathbb{S}$ be the Serre functor of the homologically smooth partially wrapped Fukaya category $\mathcal{W}(\Sigma, M, \eta)$. Then
  
        (1) Both $\uSdim$ and $\lSdim$ are rational numbers. 

        (2) The upper Serre dimension $\uSdim$ is non-negative. 

        (3) One has an upper bound $\mathrm{Rdim}\le \uSdim$ for the Rouquier dimension.
\end{corollaryintro}

Now we remark on the connection of entropy and the \emph{global dimension} $\gldim \calt$ of a triangulated category $\calt$ with Bridgeland stability conditions, which is defined in terms of stability conditions in  \cite{IQ23,Q23}. By \cite[Theorem 4.2]{KOT21}, for the perfect derived category $\rm{per}(A)$ of a smooth proper dg algebra $A$, it is shown that the global dimension of $\rm{per}(A)$ is bounded below by the upper Serre dimension.
Thus, it follows that for a partially wrapped Fukaya category of a surface with stops $\mathcal{W}(\Sigma, M, \eta)$,  we have that 
$$ 1-\min \Omega \leq \gldim \mathcal{W}(\Sigma, M, \eta).$$ 
Moreover, by  \cite[Theorem 6.4]{Q25} and Theorem \ref{thm:main2} in the case that $\Sigma$ is an annulus, the global dimension is given by the upper Serre dimension, that is 
$$  \uSdim \mathcal{W}(\Sigma, M, \eta) =  \gldim \mathcal{W}(\Sigma, M, \eta)   .$$
One could ask if the above equality holds for any surface $\Sigma$, for more details, see  \cite{Q25}.

\medskip

The authors of~\cite{DHKK14} ask (Question 4.1) whether the entropy is always algebraic: that is, whether the graph of the function $t\mapsto H_t$ is an algebraic curve over $\mathbb Q$ in coordinates $(e^t, e^{H_t})$. Note that the graph of the linear function $t\mapsto \alpha t$ is an algebraic curve in coordinates $(e^t, e^{\alpha t})$ if and only if $\alpha$ is rational. We are not sure if the union of several pieces of algebraic curves is an algebraic curve or not; however, we find it interesting to understand when the entropy of the Serre functor is given by a linear, and not a piecewise linear function. Here is the answer for 
the case of a partially wrapped Fukaya category of a surface with stops.

\begin{theoremintro}
[Theorem~\ref{thm:Main3Section2}]
\label{thm:main3}
Suppose $\calt$ is a homologically smooth partially wrapped Fukaya category of a surface $\Sigma$ with stops which is not a disc (or $\calt =\dba$ for a homologically smooth graded gentle algebra $A$ which is not derived equivalent to a hereditary algebra of type $\mathbb{A}$). Let  $H_t(\mathbb{S})$ be the entropy of the Serre functor $\mathbb{S}$  of $\calt$. Then the following statements are equivalent:

\begin{enumerate}
    \item  $H_t(\mathbb{S})$ is linear;
    \item[(1')] $\lSdim \calt=\uSdim\calt$;
    \item $H_t(\mathbb{S})=t$;
    \item[(2')]  $\lSdim \calt=\uSdim\calt=1$;
    \item $\omega_i=0$ for all boundary components of $\Sigma$;
    \item $m_i=n_i$ for all pairs $(m_i,n_i)$ in the AG-invariant;
    \item $\Sigma$ is an annulus and the winding number along the equator is zero; 
    \item $A$ is derived equivalent to a hereditary algebra of affine type $\Tilde{\mathbb A}$ with trivial grading;
    \item $A$ is derived equivalent to a $d$-representation infinite algebra  in the sense of \cite{HIO14}. 
    \label{claim4}
\end{enumerate}
\end{theoremintro}

We thank Calvin Pfeifer for pointing out statement \eqref{claim4} to us. Note that it is the $d$-representation infinite analogue of the results in \cite{HauglandJacobsenSchroll} and so the only higher representation finite and infinite gentle algebras are the algebras of Dynkin and affine type $\Tilde{\mathbb A}$. 
For the case of a disc, which corresponds to graded gentle algebras derived equivalent to hereditary algebras of Dynkin type $\mathbb{A}$,  the entropy of the Serre functor has been shown to be a linear function in \cite{DHKK14}.

\medskip
Finally, for $A$ a finite dimensional gentle algebra, we write   $[\SSS]$ for the endomorphism of $\calk_0 (\dba)$ induced by the Serre functor $\SSS$ and let $\rho([\SSS])$ be the spectral radius of~$[\SSS]$. Then following  \cite{G87,G03,Y87}, see also \cite{Og14}, $\log \rho([\SSS])$ could be considered to be the topological entropy of $\SSS$ and we show that a Gromov--Yomdin-like equality holds. Namely, denoting the categorical entropy of $\SSS$ by $H_{\rm{cat}}(\SSS):=H_0(\SSS)$, we show the following. 

\begin{theoremintro}
[Theorem~\ref{thm:categorical vs topological entropy}]\label{thm:main4}
Let $A$ be a finite dimensional gentle algebra of finite global dimension with Serre functor $\SSS$ in $\dba$ and let $[\SSS]: \calk_0 (\dba) \to \calk_0 (\dba)$ be the induced map on the  Grothendieck group of $\dba$. Then 
$$H_{\rm{cat}}(\SSS) = \log \rho([\SSS]).$$
\end{theoremintro}

The equality in the above theorem can be viewed as
a categorical generalization of the fundamental theorem on topological entropy due to
Gromov and Yomdin \cite{G87,G03,Y87}, which is conjectured to be true for autoequivalences on the derived category of a smooth projective variety, see details in \cite[Conjecture 5.3]{KT19} and \cite[Conjecture 2.14]{K17}.
This conjecture holds for some cases, see for example \cite{K17, KT19, KST20, F18b, O18, Y20, K23, LP23, H24} but is proved not to hold in general, with counter-examples given in  \cite{F18a}, see also \cite{O20, M21, BK23a, BK23b, KO23}. 

Upon finalizing this paper, we learned that at the same time Tomasz Ciborski has calculated the categorical and polynomial entropy of autoequivalences of derived discrete algebras \cite{C25}.  

\section*{Acknowledgments}
This paper is supported by the Fundamental Research Funds for the Central Universities (No. GK202403003), the NSF of China (Grant No. 12271321), the UKRI Horizon Europe guarantee award Motivic invariants and birational geometry of simple normal crossing degenerations EP/Z000955/1, and the DFG through the project SFB/TRR 191 Symplectic Structures in Geometry, Algebra and Dynamics (Projektnummer 281071066-TRR 191). 
The authors would like to thank Nathan Broomhead for very useful discussions,  Yang Han for valuable suggestions on the references for categorical entropy, and  Calvin Pfeifer for helpful discussions and feedback on an earlier version of the paper.

\section{Preliminaries}\label{Preliminaries}
\subsection{Entropy}
The entropy of an exact endofunctor $F$ of a triangulated category ${\calt}$ is introduced in \cite{DHKK14}. It is the asymptotically exponential growth rate of the complexity of $F$. In the following, we begin by briefly recalling some basic definitions.  Let $E_1,E_2$ be two objects in ${\calt}$, and $[1]$ be the shift functor in $\calt$.
The {\it complexity} of $E_2$ with respect to $E_1$ is the function $\zd_t(E_1,E_2): \mathbb{R}\to \mathbb{R}_{\ge 0}\cup\{\infty\}$ in the real variable $t$ given by
$$\zd_t(E_1,E_2):=\mathrm{inf}\left\{ \sum\limits_{i=1}^m e^{d_it}\ \left|\
\xymatrix@=3pt{
T_0 \ar[rr] & & T_1 \ar[dl] & \cdots & T_{m-1} \ar[rr] && T_m \ar[dl] \\
& E_1[d_1] \ar@{.>}[ul] &   &        && E_1[d_m] \ar@{.>}[ul] & }
\right.\right\}$$
where $T_0=0, T_m=E_2\oplus E'_2$ for some $E'_2\in {\calt},$ and $$T_{i-1}\to T_i\to E_1[d_i]\to T_{i-1}[1],\linebreak 1\le i\le m$$ 
are triangles in ${\calt}$.
By convention, $\zd_t(E_1,E_2):=0$ if $E_2=0$, and $\zd_t(E_1,E_2):=\infty$ if and only if $E_2$ is not in the thick triangulated subcategory $\thick_\calt E_1$ of ${\calt}$ generated by $E_1$.
Note that $\zd_0(E_1,E_2)$ is the least number of steps required to build $E_2$ from objects in $\{E_1[n]\ |\ n\in\mathbb{Z}\}$.

Usually, one considers a  triangulated category ${\calt}$ with a \emph{(split=classical) generator} $G$, that is, an object such that ${\calt}=\thick_{\calt} G$, or equivalently, for every object $E\in {\calt}$ there is an object $E'\in {\calt}$ and a tower of triangles in ${\calt}$
$$\xymatrix@=4pt{
0=T_0 \ar[rr] && T_1 \ar[rr] \ar[dl] && T_2 \ar[dl] & \cdots & T_{m-1} \ar[rr] && T_m=E\oplus E' \ar[dl] \\
& G[d_1] \ar@{.>}[ul] && G[d_2] \ar@{.>}[ul] &&&& G[d_m] \ar@{.>}[ul] & }$$
with $m\in\mathbb{Z}_{\ge 0}$ and $d_i\in\Z$ for all $1\le i\le m$. 

\begin{definition}\cite[Definition 2.5]{DHKK14}
\label{definition:entropy}
Let ${\calt}$ be a triangulated category with a generator $G$, and $F$ an endofunctor of ${\calt}$. The {\it entropy} of $F$ is the function $H_t(F):\mathbb{R}\to \mathbb{R}\cup\{-\infty\}$ in the real variable $t$ given by
$$H_t(F):=\lim\limits_{N\to\infty}\frac{1}{N}\log\zd_t(G,F^N(G)).$$    
\end{definition}

By~\cite[Lemma 2.5]{DHKK14}, entropy is well-defined as it does not depend on the choice of generator $G$. Moreover, for a triangulated category ${\calt}$ admitting a dg enhancement, the entropy can be calculated  as follows. 

\begin{theorem}\label{thm:dhkk1}\cite[Theorem 2.7]{DHKK14}
Let $\mathcal{T}$ be a triangulated category admitting a dg enhancement given by a saturated dg category, and $F$ an endofunctor of $\mathcal{T}$ admitting a dg lift. Then we have $H_t(F)=h_t(F)$, where
\begin{equation}\label{formula:hom}h_t(F)=\lim\limits_{N\to\infty}\frac{1}{N}\log\sum_{n\in\mathbb{Z}}\ \dim_k\Hom_{\calt}(G,F^N(G)[n])\cdot e^{-nt}.
\end{equation}
\end{theorem}

In this paper, we will consider the entropy $H_t(\SSS)$ of the Serre functor $\mathbb{S}$ of the homologically smooth and proper partially wrapped Fukaya category of a graded surface with stops, which is a triangulated category admitting a dg enhancement as in Theorem \ref{thm:dhkk1}. Thus,  if $\calt$ is a triangulated category with such a dg enhancement and a Serre functor, the following properties hold for an endofunctor $F$ of $\calt$. 

\begin{lemma} \label{lem:pre}
(1) \cite[Section 2.2]{DHKK14} $H_t(F^m)=m\cdot H_t(F)$ for any $m\in\mathbb{Z}_{>0}$;

(2) \cite[Lemma 2.7]{KST20} $H_t(F[m])=H_t(F)+mt$ for any $m\in\mathbb{Z}$;

(3) \cite[Lemma 2.11]{FFO21}  $H_t(F^{-1})=H_{-t}(F)$ if $F$ is an autoequivalence of $\mathcal{T}$.
\end{lemma}

\subsection{Graded marked surfaces and graded gentle algebras}\label{subsection:marked surfaces}

\begin{definition}
\label{definition:marked surface}
A \emph{marked surface} is a pair $(\Sigma,\calm)$, where
  
  (1) $\Sigma$ is a smooth real compact oriented surface with non-empty boundary $\partial=\partial\Sigma$ whose
  connected components are $\partial_1, \ldots, \partial_b$.
 
  (2) $\calm = \calm_{\gpoint} \sqcup \calm_{\rpoint}$ is a finite set of \emph{marked points} on $\Sigma$. The elements of~$\calm_{\gpoint}$  are on the boundary of $\cals$ and the elements of~$\calm_{\rpoint}$ can be both on the boundary or in the interior of $\Sigma$. Each boundary component contains at least one marked point, and the $\gpoint$-points and~$\rpoint$-points are alternating on them.   
  Elements of $\calm_{\gpoint}$ and $\calm_{\rpoint}$ will be respectively denoted by  symbols~$\gpoint$ and~$\rpoint$. 
\end{definition}

The marked points in $\calm_{\rpoint}$ correspond to the \emph{stops} used in \cite{HKK17}.
We only consider graded gentle algebras that are proper in this paper. Thus, there are no $\gpoint$-point in the interior of the surface.

\begin{definition}
\label{definition:arcs}
Let $(\Sigma,\calm)$ be a marked surface. 
\begin{enumerate}[\rm(1)]
 \item An \emph{arc} is a smooth map $[0,1]\to \Sigma$, with endpoints in~$\calm$ and not containing other marked points. 
 \item An \emph{$\gpoint$-arc} is an arc with endpoints in~$\calm_{\gpoint}$.
 \item An \emph{$\rpoint$-arc} is an arc with endpoints in $\calm_{\rpoint}$.
 \item A \emph{loop} is an arc whose endpoints coincide.
 \item A \emph{simple arc} is an arc without interior self-intersections.
 \end{enumerate}
\end{definition}

In order for some definitions and notations to be well-defined in the case of a loop, we will treat the unique endpoint of a loop as two distinct endpoints.
All arcs are considered up to isotopy in the class of arcs, and up to change of orientation, arcs are assumed to be non-contractible. All intersections of arcs are assumed to be transversal. Furthermore, we assume that all arcs are in \emph{minimal position}: that is, the number of intersections is the minimal one possible in their isotopy classes and there are no triple intersections. We fix orientation of the surface and the induced orientation of the boundary: when we follow the boundary in positive direction, the interior of the surface is on the left.

\begin{definition}\label{definition:addmissable dissections in prelimilary}
Let $(\Sigma,\calm)$ be a marked surface.

(1) A collection of simple arcs with endpoints in $\calm_{\gpoint}$ is called an \emph{admissible $\gpoint$-dissection}, if the arcs have no interior intersections and they cut the surface into polygons, each of which contains exactly one marked point from $\calm_{\rpoint}$.
 
(2) A collection of simple arcs with endpoints in $\calm_{\rpoint}$ is called an \emph{admissible $\rpoint$-dissection}, if the arcs have no interior intersections and they cut the surface into polygons each of which contains exactly one marked point from $\calm_{\gpoint}$.
\end{definition}

An admissible $\gpoint$-dissection will be denoted by $\zD$. Then it is not hard to see that (see \cite{OPS18}) for each $\zD$, there is a unique (up to homotopy) dual $\rpoint$-admissible dissection on $(\Sigma,\calm)$, which we will denote by $\zD^*$, such that each arc $\ell^*$ in $\zD^*$ intersects exactly one arc $\ell$ of $\zD$, and vice verca.

\begin{definition}\label{definition:addmissable dissections in prelimilary2}
Let $\zD$ and $\zD^*$ be admissible dissections on a marked surface $(\cals,\calm)$.

  (1) Let $q$ be a common endpoint of arcs $\ell_i, \ell_j$ in $\zD$, an \emph{oriented intersection} from $\ell_i$ to~$\ell_j$ at $q$ is an anti-clockwise angle from $\ell_i$ to $\ell_j$  based at $q$ such that the angle is in the interior of the surface. We call an oriented intersection a \emph{minimal oriented intersection} if it is not a composition of two oriented intersections of arcs in $\zD$.
  
  (2) Let $q$ be a common endpoint of arcs $\ell_i^*, \ell_j^*$ in $\zD^*$, an \emph{oriented intersection} from $\ell_i^*$ to~$\ell_j^*$ at $q$ is a clockwise angle from $\ell_i^*$ to $\ell_j^*$ based at $q$ such that the angle is in the interior of the surface. We call an oriented intersection a \emph{minimal oriented intersection} if it is not a composition of two oriented intersections of arcs in $\zD^*$.
\end{definition}

In this paper, we define a graded marked surface as follows.

\begin{definition}
\label{definition:graded marked surface}
A \emph{graded marked surface} is a quadruple 
$(\Sigma, \calm, \zD^*, G)$
where $(\Sigma, \calm)$ is a marked surface, $\zD^*$ is an admissible $\rpoint$-dissection, and   $G$ is a \emph{grading}, which is given by a map from the set of minimal oriented intersections in $\zD^*$ to $\mathbb{Z}$, sending a minimal oriented intersection   $\aaa$ to an integer $G(\aaa) $. 
\end{definition}

In \cite{HKK17}, a graded surface is a triple $(\Sigma, \calm, \eta)$, where $\eta$ is a line field on $\Sigma$ with $\calm$ as a set of stops. The line field induces a grading on each arc in an admissible dissection~$\zD$ (and~$\zD^*$), as well as a grading on each intersection of the arcs in $\zD$ (and $\zD^*$).
These two ways to define a grading on a surface are, in fact, equivalent. Indeed, it is shown in \cite{LP20} that a quadruple $(\Sigma, \calm, \zD^*, G)$ induces a line field $\eta$ on $\Sigma$ and thus gives rise to a triple $(\Sigma, \calm, \eta)$ which is unique up to diffeomorphism. 
To be more accurate, in \cite{LP20} it is assumed that there are no $\rpoint$-points in the interior of $\Sigma$, but the surface might have $\gpoint$-points in the interior, corresponding to unstopped boundary components. If the surface has interior $\rpoint$-points, one should work with the line field obtained from the surface $\Sigma'$ where in the surface $\Sigma$ one  replaces the interior $\rpoint$-points with interior $\gpoint$-points and then considers the full subcategory of the partially wrapped Fukaya category of $\Sigma'$ induced by the  closed curves of winding number zero (except for those around the interior $\gpoint$-points) and the $\gpoint$-arcs not connected to the interior $\gpoint$-points, for more details see \cite{HKK17, BSW24}.

Fix a field $k$.
Let $Q=(Q_0,Q_1)$ be a finite quiver, where $Q_0$ is the set of vertices and~$Q_1$ is the set of arrows. Let $k Q$ be the path algebra of $Q$, where we read composition of arrows from left to right, and $I\subset k Q$ be an ideal.
\begin{definition}
The path algebra with relations $k Q/I$ is called \emph{gentle} if the following conditions hold
\begin{enumerate}
    \item for any vertex $i\in Q_0$, there are at most two arrows starting in $i$ and 
     at most two arrows ending in $i$,
    \item for any arrow $\aaa\in Q_1$ from $i$ to $j$
    \begin{enumerate}
        \item if there are two arrows $\bbb_1,\bbb_2$ in $Q_1$ ending at $i$ then exactly one of  $\bbb_1\aaa,\bbb_2\aaa$ is in $I$,
        \item if there are two arrows $\ccc_1,\ccc_2$ in $Q_1$ starting at $j$ then exactly one of  $\aaa\ccc_1,\aaa\ccc_2$ is in $I$,
    \end{enumerate}
    \item $I$ is generated by paths of length $2$.
\end{enumerate}
Additionally, we will always assume that a gentle algebra is finite dimensional and connected.

A \emph{grading} on a gentle algebra is given by a function $G\colon Q_1\to \Z$. Such a grading turns~$k Q$ into a $\Z$-graded algebra.
\end{definition}

There is a bijection between the set of homeomorphism classes of graded marked surfaces with admissible $\rpoint$-dissections and the set of isomorphism classes of graded gentle algebras. Below, we recall the constructions. First, we construct a graded gentle algebra from a graded marked surface with dissection.

\begin{definition}\label{definition:gentle algebra from dissection}
For a graded marked surface $(\Sigma,M,\zD^*, G)$ and a field $k$, 
consider a graded $k$-algebra $A$ as the quotient of the path algebra $kQ$ of a quiver $Q$ by an ideal $I$ defined as follows:

	(1) The vertices of $Q$ are given by the arcs in $\zD^*$.
	
    (2) Each minimal oriented intersection $\aaa$ from $\ell^*_i$ to $\ell^*_j$ gives rise to an arrow from $\ell^*_i$ to $\ell^*_j$, which is still denoted by $\aaa$.
	
    (3) The ideal $I$ is generated by paths $\aaa\bbb:\ell^*_i\rightarrow \ell^*_j\rightarrow \ell^*_k$, where the common endpoints of $\ell^*_i$ and $\ell^*_j$, and the common endpoints of $\ell^*_j$ and $\ell^*_k$  giving rise to $\aaa$ and $\bbb$, respectively, all coincide.
     
     (4) The grading of $A$ is induced by the grading $G$, by associating to  each arrow $\aaa$  the corresponding integer $|\aaa|:=G(\aaa)$. 
\end{definition}

Now we construct a graded marked surface with dissection from a graded gentle algebra~$A$.

\begin{definition}
\label{def:Sigma-from-A}    
For a graded gentle algebra $A$, construct its geometric model $(\Sigma,M,\zD^*, G)$ as follows.
For any maximal path $a_1\to a_2\to\ldots\to a_n$ in $A$ take an $(n+1)$-gon, put $\rpoint$-points in all vertices, put a $\gpoint$-point on one edge, label all other edges by $a_1,\ldots,a_n$ going anti-clockwise from the $\gpoint$-point. Connect this point by non-intersecting curves to edges $a_1,\ldots,a_n$. For any vertex $a$ in $Q$, there is now one or two edges with the label $a$. If there are two, glue these edges, preserving orientation and matching the curves drawn on the two sides. If there is just one edge labeled $a$, then either $a$ is an endpoint in $Q$ or there is a unique arrow $\xxx$ starting at $a$, a unique arrow~$\yyy$ ending at $a$, and $\yyy\xxx\ne 0$ in $A$. In these two cases, take a $2$-gon, mark both vertices as $\rpoint$-points, label one edge with $a$, mark a $\gpoint$-point on the other edge, connect the marked point by a curve with the opposite edge $a$, and glue this edge to another edge labeled~$a$ as before. As a result, one gets a connected oriented surface with a marking, an admissible $\rpoint$-dissection (given by the glued edges), and its dual $\gpoint$-dissection (given by the drawn curves). The boundary of this surface is formed by the non-glued edges of polygons. Note that the arcs forming dissections correspond to the vertices of $Q$. 
Introduce a grading on the surface as follows:  any minimal oriented intersection between arcs forming $\rpoint$-dissection comes from an arrow $\xxx$ from $a_i\to a_{i+1}$ in $Q$, let $G(\xxx)$ be the degree of this arrow. 
\end{definition}

It is proved (see \cite{OPS18}, \cite{LP20}, as well as in the ungraded case \cite{S15}, \cite{BC21})  that these two constructions are inverse to each other and establish a bijection between graded marked surfaces with dissection and graded gentle algebras.

\subsection{Geometric models for the derived categories of graded gentle algebras}\label{subsection: derived categories and derived categories}

Let $A$ be a $\mathbb{Z}$-graded proper gentle algebra and let $\dba$ be the perfect derived category of dg-modules over $A$ viewed as a dg-algebra with zero differential.

In this subsection, we recall the geometric model for $\dba$ from \cite{OPS18}. In the following, let $(\Sigma,\calm,\zD^*,G)$ be the graded marked surface associated with $A$.
It is proved in \cite{OPS18} (see also in \cite{HKK17}) that the classification of the indecomposable objects in $\dba$ for gentle algebra $A$ with trivial grading can be generalized to the case of graded gentle algebras, that is, each indecomposable object in $\dba$ is either a \emph{homotopy string object} or a \emph{homotopy band object}. In this paper, we only consider generators of $\dba$ that are given by homotopy string objects. In the associated surface, these correspond to graded $\gpoint$-arcs.

Let $\za$ be an $\gpoint$-arc on $(\Sigma,\calm)$. After choosing a direction of $\za$, denote by $\bbp_0, \bbp_1,\cdots, \bbp_{n}$ the ordered set of polygons of $\zD^*$ that $\za$ successively crosses. Denote by $\zD^*\cap \za:=\{\ell^*_1,\cdots, \ell^*_{n}\}$ the ordered set of arcs in $\zD^*$ that successively intersect  $\za$ such that $\ell^*_i$ belongs to $\bbp_{i-1}$ and $\bbp_{i}$ for each $1 \leq i \leq n$.
Remark that we consider $\zD^*\cap \za$ as an ordered multi-set, that is, we distinguish the different occurrences of the same arc.

Note that for each $1\leq i \leq n-1$, there is a canonical non-zero path $\sigma_i$ in $A$ from $\ell^*_{i}$ to $\ell^*_{i+1}$ (or from $\ell^*_{i+1}$ to $\ell^*_{i}$), which is the composition of the arrows arising from the minimal oriented intersections in $\bbp_i$, provided that the unique $\gpoint$-point $p$ in $\bbp_i$ is to the left (respectively right) of $\za$.

A \emph{grading} of $\za$ is a function $f: \zD^*\cap \za \longrightarrow \mathbb{Z}$ that satisfies
\begin{equation}\label{equ:grading}
f(\ell^*_{i+1}) = \begin{cases}
              {f(\ell^*_i) -|\sigma_i|+1} & \textrm{if {$p$ is to the left of $\za$} in $\bbp_i$;}  \\
              f(\ell^*_i)+|\sigma_i|- 1 & \textrm{if {$p$ is to the right of $\za$} in $\bbp_i$,}
             \end{cases}
\end{equation}

for each $1\leq i \leq n-1$, where the grading $|\sigma_i|$ is determined by the grading map $G$, see Definition \ref{definition:gentle algebra from dissection}(4). We call the pair $(\za,f)$ a \emph{graded arc} on  $(\cals,\calm,\zD^*,G)$.
It is not hard to see that the definition of the grading is independent of the choice of the direction of the arc.

For a graded arc $(\za,f)$ and $n\in\Z$ we define its \emph{shift} as $(\za,f)[n]:=(\za, f-n)$.

Let $(\za,f)$ be a graded arc on $(\cals,\calm,\zD^*,G)$ and $A$ be the corresponding graded gentle algebra. 
In the following, we construct a perfect dg module $\P_{(\za,f)}$ over $A$.
Note that an arc $\ell^*\in\zD^*$ corresponds to a vertex in the quiver $Q=(Q_0,Q_1)$ of $A$. 
Denote by $P_{\ell^*}=A\cdot e_{\ell^*}$ the indecomposable projective graded left $A$-module  given by the vertex associated to $\ell^*$. Consider
the direct sum
\begin{equation}
\label{eq_P-direct-sum}
\P_{(\za,f)}=\bigoplus_{i=1}^nP_{\ell^*_i}[-d_i],
\end{equation}
where $\zD^*\cap \za=\{\ell^*_1,\cdots, \ell^*_{n}\}$ is the ordered set of arcs in $\zD^*$ that successively intersect~$\za$ and $d_i=f(\ell^*_i)$. 
For any $i=1,\ldots,n-1$ there is a path $\sigma_i$ in $A$ from $e_i$ to $e_{i+1}$ (respectively, from $e_{i+1}$ to $e_{i}$), which defines a homogeneous morphism $P_{\ell^*_{i}}[-d_i]\longrightarrow P_{\ell^*_{i+1}}[-d_{i+1}]$ (respectively, $P_{\ell^*_{i+1}}[-d_{i+1}]\longrightarrow P_{\ell^*_{i}}[-d_{i}]$) of degree $1$ if the $\gpoint$-point in $\bbp_i$ is to the left (respectively, right) of $\za$. We take all these morphisms together to define a differential on $\P_{(\za,f)}$.  In this way, we obtain a dg module $\P_{(\za,f)}$ over $A$, which will be viewed as an object in $\dba$ and which we refer to as a \emph{homotopy string object}.

Similarly to the above, to any \emph{graded closed curve} on $(\cals,\calm,\zD^*,G)$ one can associate a family of objects  in $\dba$ (parametrised by indecomposable $k[t,t^{-1}]$-modules) called \emph{homotopy band objects}. We will not need these objects; therefore refrain from further details.

\begin{proposition}[\cite{BM03}, {\cite[Theorem 4.3]{HKK17}}, {\cite[Theorem 2.13]{OPS18}}]
Indecomposable objects in $\dba$ are exactly homotopy string objects and homotopy band objects.
\end{proposition}

\begin{definition} Given a complex   $\P_{(\za,f)}$ as above, we call the minimal $f(\ell^*)$ the \emph{left length}, and call the maximal $f(\ell^*)$ the \emph{right length} of $\P_{(\za,f)}$. The \emph{length} of $\P_{(\za,f)}$ is defined as the right length minus the left length.
\end{definition}

Note that the length of a complex is always non-negative, while the left length and the right length of a complex are allowed to be any integer.

Furthermore, based on the results in \cite{ALP16}, it is proved in \cite[Theorem 3.3, Remark 3.8]{OPS18} that the morphisms in $\dba$ can be interpreted as the oriented intersections of graded $\gpoint$-arcs on the surface, which will be restated in the following. We start by introducing the notion of an oriented intersection for ungraded arcs. 

\begin{definition}
Let $\za$ and $\zb$ be two $\gpoint$-arcs on $(\cals,\calm)$ which intersect at a point $q$. An \emph{oriented intersection} from $\za$ to $\zb$ at $q$ is the anti-clockwise angle from $\za$ to $\zb$ based at $q$ such that the angle is in the interior of the surface, with the convention that if $q$ is an interior point of the surface then the opposite angles are considered equivalent.
\end{definition}

There are two possibilities for the position of $q$: $q$ is a boundary marked point or an interior point, see  Figure \ref{interior intersection}. A boundary intersection gives rise to only one oriented intersection $\aaa$ as in the left picture in Figure \ref{interior intersection} where $\aaa$ corresponds to an oriented intersection from $\alpha$ to $\beta$, while an interior intersection gives rise to two oriented intersections, namely an oriented intersection  $\aaa$ from $\alpha$ to $\beta$ and an oriented intersection $\bbb$ from $\beta$ to~$\alpha$ as in the right picture in Figure \ref{interior intersection}.

In order to consider the morphisms arising from an oriented intersection $\aaa$ from $\za$ to $\zb$ at $q$, we fix  gradings 
$f$ and $g$ of $\za$ and $\zb$ respectively and define the \emph{degree} of an oriented intersection as follows.

Consider the polygon $\bbp$ containing $q$, and choose arcs $\ell^*_1$ and $\ell^*_2$ in $\zD^*$  where $\za$ and $\zb$ respectively intersect the boundary of $\bbp$ such that the following condition is satisfied:  $\ell^*_1$ comes before $\ell^*_2$ in the anti-clockwise direction on the boundary of $\bbp$, if starting from the marked $\gpoint$-point in $\bbp$. See Figure~\ref{Figure2} for the possible cases, depending on whether $\alpha$ and $\beta$ have endpoints in $\bbp$ or not.
Note that the choice of $\ell^*_1$ and $\ell^*_2$ may be non-unique, but the degree does not depend on the choice.  

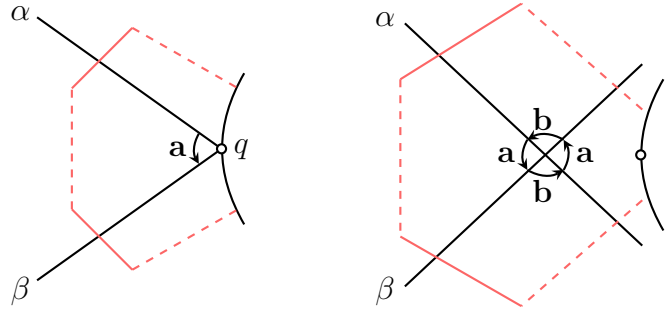
\begin{figure}[ht]
		{
\begin{tikzpicture}[>=stealth,scale=0.4]
			\draw[thick,-] (7.95,2.5) -- (1.85,6.85);
			\draw[thick,-] (7.95,2.5) -- (1.85,-1.85);
		
			\draw (1.3,7) node {$\za$};
            \draw [bend right,thick] (8.7,5) to (8.7,0);

			\draw[thick,red!60] (3,4.5) -- (5,6.5);

			\draw[thick,red!60] (3,.5) -- (5,-1.5);

			\draw[dashed,thick,red!60] (3,.5) -- (3,4.5);

			\draw[dashed,thick,red!60] (5,6.5) -- (8.5,4.5);
			\draw[dashed,thick,red!60] (5,-1.5) -- (8.5,.5);
			\draw (1.3,-2.2) node {$\zb$};
			
			\draw (8.6,2.5) node {$q$};
            \draw [bend left,thick,<-] (7.2,2) to (7.2,3);
			\draw (6.5,2.5) node {$\aaa$};
            \draw[thick,fill=white] (7.95,2.5) circle (0.15);
			\draw (12,0) node {};
\end{tikzpicture}
\begin{tikzpicture}[>=stealth,scale=0.4]
			\draw[ thick,-] (8,-.5) -- (0.15,6.85);
			\draw[ thick,-] (8,5.5) -- (0.15,-1.85);

			\draw (-.5,7) node {$\za$};

            \draw [bend right,thick] (8.7,5) to (8.7,0);

%
			\draw[dashed,thick,red!60] (0,5) -- (0,-.2);

			\draw[thick,red!60] (0,5) -- (4,7.5);
			\draw[thick,red!60] (0,-.2) -- (4,-2.5);

			\draw[dashed,thick,red!60] (8,4) -- (4,7.5);
			\draw[dashed,thick,red!60] (8,1) -- (4,-2.5);

			\draw (-.5,-2) node {$\zb$};
			

            \draw [bend left,thick,<-] (4.2,2) to (4.2,3);
			\draw (3.5,2.5) node {$\aaa$};
            \draw [bend left,thick,<-] (5.4,3) to (5.4,2);
			\draw (6.1,2.5) node {$\aaa$};
            \draw [bend left,thick,<-] (4.2,3) to (5.4,3);
			\draw (4.7,3.7) node {$\bbb$};
            \draw [bend left,thick,<-] (5.4,2) to (4.2,2);
			\draw (4.7,1.3) node {$\bbb$};

            \draw[thick,fill=white] (7.95,2.5) circle (0.15);

			\end{tikzpicture}}
			\caption{Two types of oriented intersections between $\gpoint$-arcs  }
            \label{interior intersection}
\end{figure}

\begin{figure}[ht]       
		\scalebox{.65}{
		\begin{tabular}{c c c c c}
			{\begin{tikzpicture} 
				
				\foreach \u in {1,2,4} 
				\draw[dashed,thick, red!60] ({2*cos(360/4*\u-360/16)},{2*sin(360/4*\u-360/16)}) arc ({360/4*\u-360/16}:{360/4*\u+360/16}:2);

                \foreach \u in {3} 
				\draw[line width=1] ({2*cos(360/4*\u-360/16)},{2*sin(360/4*\u-360/16)}) arc ({360/4*\u-360/16}:{360/4*\u+360/16}:2);
				
				\foreach \u in {1,2} 
				\draw[thick] ({2*cos(360/8+360/4*\u)}, {2*sin(360/8+360/4*\u)})--({2*cos(180+360/8+360/4*\u)}, {2*sin(180+360/8+360/4*\u)}); 
				
				\foreach \u in {1,...,4} 
				\draw [line width=1, color=red] plot  [smooth, tension=1] coordinates {   ({2*cos(360/8+360/4*\u+360/16)}, {2*sin(360/8+360/4*\u+360/16)}) ({1.6*cos(360/8+360/4*\u)}, {1.6*sin(360/8+360/4*\u)}) ({2*cos(360/8+360/4*\u-360/16)}, {2*sin(360/8+360/4*\u-360/16)})};
				
				\draw ({2.3*cos(360/8)}, {2.3*sin(360/8)}) node {$\beta$};
				\draw ({2.3*cos(-360/8)}, {2.3*sin(-360/8)}) node {$\alpha$};

                \draw[red] ({1.5*cos(25)}, {1.5*sin(25)}) node {$\ell^*_2$};
				\draw[red] ({1.5*cos(-25)}, {1.5*sin(-25)}) node {$\ell^*_1$};
                \draw[red] ({1.5*cos(203)}, {1.5*sin(203)}) node {($\ell^*_2$)};
				\draw[red] ({1.5*cos(160)}, {1.5*sin(160)}) node {($\ell^*_1$)};
                
				\draw[fill=white] ({2*cos(270)},{2*sin(270)}) circle (2.5pt);

				\end{tikzpicture}
			} 
            &
            {\begin{tikzpicture} 
				
				\foreach \u in {1,2,4} 
				\draw[dashed,thick, red!60] ({2*cos(360/4*\u-360/16)},{2*sin(360/4*\u-360/16)}) arc ({360/4*\u-360/16}:{360/4*\u+360/16}:2);

                \foreach \u in {3} 
				\draw[line width=1] ({2*cos(360/4*\u-360/16)},{2*sin(360/4*\u-360/16)}) arc ({360/4*\u-360/16}:{360/4*\u+360/16}:2);
				
				\foreach \u in {1,2} 
				\draw[thick] ({2*cos(360/8+360/4*\u)}, {2*sin(360/8+360/4*\u)})--({2*cos(180+360/8+360/4*\u)}, {2*sin(180+360/8+360/4*\u)}); 
				
				\foreach \u in {1,...,4} 
				\draw [line width=1, color=red] plot  [smooth, tension=1] coordinates {   ({2*cos(360/8+360/4*\u+360/16)}, {2*sin(360/8+360/4*\u+360/16)}) ({1.6*cos(360/8+360/4*\u)}, {1.6*sin(360/8+360/4*\u)}) ({2*cos(360/8+360/4*\u-360/16)}, {2*sin(360/8+360/4*\u-360/16)})};
				
				\draw ({2.3*cos(360/8)}, {2.3*sin(360/8)}) node {$\alpha$};
				\draw ({2.3*cos(-360/8)}, {2.3*sin(-360/8)}) node {$\beta$};

                \draw[red] ({1.5*cos(25)}, {1.5*sin(25)}) node {$\ell^*_1$};
				\draw[red] ({1.5*cos(157)}, {1.5*sin(157)}) node {$\ell^*_2$};
                
				\draw[fill=white] ({2*cos(270)},{2*sin(270)}) circle (2.5pt);

				\end{tikzpicture}
			} 
            
			&
            {\begin{tikzpicture} 
				
				\foreach \u in {1,2,4} 
				\draw[dashed,thick, red!60] ({2*cos(360/4*\u-360/16)},{2*sin(360/4*\u-360/16)}) arc ({360/4*\u-360/16}:{360/4*\u+360/16}:2);

                \foreach \u in {3} 
				\draw[line width=1] ({2*cos(360/4*\u-360/16)},{2*sin(360/4*\u-360/16)}) arc ({360/4*\u-360/16}:{360/4*\u+360/16}:2);
				
				\foreach \u in {1} 
				\draw[thick] ({2*cos(360/8+360/4*\u)}, {2*sin(360/8+360/4*\u)})--({2*cos(180+360/8+360/4*\u)}, {2*sin(180+360/8+360/4*\u)}); 

                \draw[thick] ({0}, {-2})--({2*cos(45)}, {2*sin(45)}); 
				
				\foreach \u in {1,...,4} 
				\draw [line width=1, color=red] plot  [smooth, tension=1] coordinates {   ({2*cos(360/8+360/4*\u+360/16)}, {2*sin(360/8+360/4*\u+360/16)}) ({1.6*cos(360/8+360/4*\u)}, {1.6*sin(360/8+360/4*\u)}) ({2*cos(360/8+360/4*\u-360/16)}, {2*sin(360/8+360/4*\u-360/16)})};
				
				\draw ({2.3*cos(360/8)}, {2.3*sin(360/8)}) node {$\alpha$};
				\draw ({2.3*cos(-360/8)}, {2.3*sin(-360/8)}) node {$\beta$};

                \draw[red] ({1.5*cos(25)}, {1.5*sin(25)}) node {$\ell^*_1$};
				\draw[red] ({1.5*cos(157)}, {1.5*sin(157)}) node {$\ell^*_2$};
                
				\draw[fill=white] ({2*cos(270)},{2*sin(270)}) circle (2.5pt);

				\end{tikzpicture}
			} 
			&
			{\begin{tikzpicture} 
				
				\foreach \u in {1,2,4} 
				\draw[dashed,thick, red!60] ({2*cos(360/4*\u-360/16)},{2*sin(360/4*\u-360/16)}) arc ({360/4*\u-360/16}:{360/4*\u+360/16}:2);

                \foreach \u in {3} 
				\draw[line width=1] ({2*cos(360/4*\u-360/16)},{2*sin(360/4*\u-360/16)}) arc ({360/4*\u-360/16}:{360/4*\u+360/16}:2);
				
				\foreach \u in {1} 
				\draw[thick] ({2*cos(360/8+360/4*\u)}, {2*sin(360/8+360/4*\u)})--({2*cos(180+360/8+360/4*\u)}, {2*sin(180+360/8+360/4*\u)}); 

                \draw[thick] ({0}, {-2})--({2*cos(45)}, {2*sin(45)}); 
				
				\foreach \u in {1,...,4} 
				\draw [line width=1, color=red] plot  [smooth, tension=1] coordinates {   ({2*cos(360/8+360/4*\u+360/16)}, {2*sin(360/8+360/4*\u+360/16)}) ({1.6*cos(360/8+360/4*\u)}, {1.6*sin(360/8+360/4*\u)}) ({2*cos(360/8+360/4*\u-360/16)}, {2*sin(360/8+360/4*\u-360/16)})};
				
				\draw ({2.3*cos(360/8)}, {2.3*sin(360/8)}) node {$\beta$};
				\draw ({2.3*cos(-360/8)}, {2.3*sin(-360/8)}) node {$\alpha$};

                \draw[red] ({1.5*cos(25)}, {1.5*sin(25)}) node {$\ell^*_2$};
				\draw[red] ({1.5*cos(-25)}, {1.5*sin(-25)}) node {$\ell^*_1$};
                
				\draw[fill=white] ({2*cos(270)},{2*sin(270)}) circle (2.5pt);

				\end{tikzpicture}
			} 
			&
			{\begin{tikzpicture} 
				
				\foreach \u in {1,2,4} 
				\draw[dashed,thick, red!60] ({2*cos(360/4*\u-360/16)},{2*sin(360/4*\u-360/16)}) arc ({360/4*\u-360/16}:{360/4*\u+360/16}:2);

                \foreach \u in {3} 
				\draw[line width=1] ({2*cos(360/4*\u-360/16)},{2*sin(360/4*\u-360/16)}) arc ({360/4*\u-360/16}:{360/4*\u+360/16}:2);

                \draw[thick] ({0}, {-2})--({2*cos(45)}, {2*sin(45)}); 
                \draw[thick] ({0}, {-2})--({2*cos(135)}, {2*sin(135)}); 
				
				\foreach \u in {1,...,4} 
				\draw [line width=1, color=red] plot  [smooth, tension=1] coordinates {   ({2*cos(360/8+360/4*\u+360/16)}, {2*sin(360/8+360/4*\u+360/16)}) ({1.6*cos(360/8+360/4*\u)}, {1.6*sin(360/8+360/4*\u)}) ({2*cos(360/8+360/4*\u-360/16)}, {2*sin(360/8+360/4*\u-360/16)})};
				
				\draw ({2.3*cos(360/8)}, {2.3*sin(360/8)}) node {$\alpha$};
				\draw ({2.3*cos(135)}, {2.3*sin(135)}) node {$\beta$};

                \draw[red] ({1.5*cos(25)}, {1.5*sin(25)}) node {$\ell^*_1$};
				\draw[red] ({1.5*cos(157)}, {1.5*sin(157)}) node {$\ell^*_2$};
                
				\draw[fill=white] ({2*cos(270)},{2*sin(270)}) circle (2.5pt);

				\end{tikzpicture}
			} 
	        \end{tabular}
	        }
\caption{Choice of the arcs $\ell^*_1$ and $\ell^*_2$ for the definition of the degree of an oriented intersection}    \label{Figure2}
\end{figure}
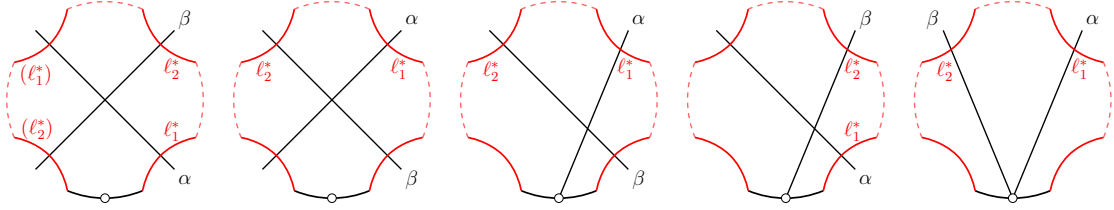

\begin{definition}
\label{def_degree-of-orint}
The \emph{degree} of the oriented intersection $\aaa$ from $(\alpha,f)$ to $(\beta,g)$ on a graded marked surface 
$(\cals,\calm,\zD^*,G)$
is
$$|\aaa|:=g(\ell^*_2)-f(\ell^*_1)+|\sigma|,$$
where $\sigma$ is a non-zero path in $A$  from  $\ell^*_1$  to $\ell^*_2$, defined as the composition of the arrows arising from 
minimal oriented intersections of arcs in $\Delta^*$ on the boundary of $\bbp$ between $\ell^*_1$  to $\ell^*_2$, and $|\sigma|$ is the degree of $\sigma$ (i.e. the sum of the corresponding values of $G$).
\end{definition}

Let $\P_{(\za,f)}$ and $\P_{(\zb,g)}$ be the two objects in $\dba$ associated with two graded $\gpoint$-arcs $(\za,f)$ and $(\zb,g)$ on the surface for some gradings $f$ and $g$. Then every oriented intersection from $\za$ to $\zb$ of degree $d$ corresponds to a morphism from $\P_{(\za,f)}$ to $\P_{(\zb,g)}[d]$ in $\dba$. One has 

\begin{proposition}\label{prop:obj-in-derived-cat}\cite[Theorem 3.3]{OPS18}
Assume that $\alpha$ and $\beta$ are in \emph{minimal position}: that is, the number of intersection points is  minimal in their isotopy classes. Then there is a bijection between a basis of the $\Hom$-space $$\Hom^\bullet(\P_{(\za,f)},\P_{(\zb,g)}):=\bigoplus\limits_{n\in \mathbb{Z}}\Hom(\P_{(\za,f)},\P_{(\zb,g)}[n]),$$
and the set of oriented intersections from $\alpha$ to $\beta$ where oriented intersections of degree $d$ correspond to  morphisms from $\P_{(\za,f)}$ to $\P_{(\zb,g)}[d]$. 
\end{proposition}

\begin{remark}
   Note that Proposition~\ref{prop:obj-in-derived-cat} applies to $\Hom$ spaces from an object $\P_{(\za,f)}$ to itself as well. For this one first replaces an arc $\alpha$ with an isotopic arc $\beta$ that is in minimal position with $\alpha$ and then applies the Proposition. In particular, exactly one of the endpoints is an oriented intersection from $\alpha$ to $\beta$, and it corresponds to the identity endomorphism.
\end{remark}

For $A$, a proper graded gentle algebra,  $\dba$ has a \emph{Serre functor} $\mathbb{S}$: that is, an auto-equivalence of $\dba$ such that for all $X,Y \in \dba$
$$\Hom(X,Y)\cong  D\Hom(Y, \mathbb{S}X),$$
where $D$ denotes the dual $k$-vector space.

In particular, 
we note that $\mathbb{S}=\tau[1]$, where $\tau$ is the Auslander--Reiten translation of $\dba$,  and   if $A$ is homologically smooth, we have $$H_t(\mathbb{S})=H_t(\tau)+t,$$
by Lemma \ref{lem:pre}(2).

Moreover, in \cite{OPS18} a geometric description of the Auslander--Reiten translate has been given as follows. 
Let $\P_{(\za,f)}$ be a homotopy string object in $\dba$. Denote by $\tau\za$ the $\gpoint$-arc obtained from $\za$ by the rotation of each of the endpoints to the next $\gpoint$-point on the same boundary component in the negative  direction: that is, so that the interior of the surface remains to the right. Then there is a canonical intersection, denoted~$q$, between $\tau \za$ and $\za$, see Figure \ref{fig:tau}. Note that this intersection may degenerate to a boundary intersection, this happens only when both  endpoints of $\za$ are located on the same boundary component.

\begin{figure}[ht]
	\centering 
	\begin{tikzpicture}[>=stealth,scale=.6]
			\draw[thick,black, fill=gray!40, line width=1.pt] (-6,0) circle (2);
			\draw[thick,black, fill=gray!40, line width=1.pt] (6,0) circle (2);
			
			\draw [line width=1.5pt, dark-green] (-4,0) to (4,0);	
			
			\draw [bend right, line width=1pt, ->] (-.5,0) to (-1.35,0.6);
			
			\draw[cyan,line width=1.5pt,black] (-6,2) to[out=20,in=120](-1,0)
			to[out=-60,in=180](2.5,-2.5)
			to[out=0,in=-160](6,-2);
			
			\draw[thick,black, fill=white] (8,0) circle (0.1);
			\draw[thick,black, fill=white] (-8,0) circle (0.1);
			\draw[thick,black, fill=white] (4,0) circle (0.1);
			\draw[thick,black, fill=white] (-4,0) circle (0.1);
			\draw[thick,black, fill=white] (-6,2) circle (0.1);
			\draw[thick,black, fill=white] (6,-2) circle (0.1);
			\draw[thick, fill=red ] (7.5,-1.35) circle (0.1);
			\draw[thick, fill=red ] (-7.5,1.35) circle (0.1);
			\draw[thick, fill=red ] (4.5,-1.35) circle (0.1);
			\draw[thick, fill=red ] (-4.5,-1.35) circle (0.1);
			\draw[thick, fill=red ] (-4.5,1.35) circle (0.1);
			\draw[thick, fill=red ] (4.5,1.35) circle (0.1);
			
			\draw (-1.4,-.5) node {$q$};
			\draw (1.5,-1.8) node {$\tau$$\za$};
	    	\draw (-.5,.9) node {$\aaa$};
			\draw (1.5,0.5) node[dark-green] {$\za$};
		\end{tikzpicture}
	\caption{Auslander--Reiten translate of an arc} 
	\label{fig:tau}
\end{figure}
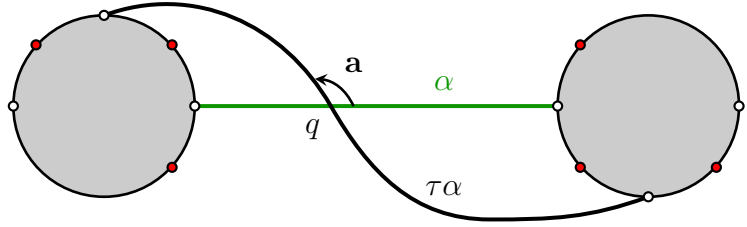
Denote by $\tau f$ the grading of $\tau \za$ such that the canonical oriented intersection $\aaa$ from $(\za,f)$ to $(\tau \za, \tau f)$ associated with $q$ has degree one. This grading can also be described as follows: the values of $\tau f$ on two intersections of $\tau \za$ with arcs in $\Delta^*$ that are the closest to $q$ are equal to the values of $f$ on the intersections of $\alpha$ with these arcs.  Building on \cite{B11, ALP16}, the  following geometric interpretation of the Auslander--Reiten translation $\tau$ in $\dba$ is given in \cite{OPS18}.

\begin{proposition}\cite[Corollary 5.4]{OPS18}\label{prop:tau}
For a homotopy string object $\P_{(\za,f)}$ in $\dba$, the object $\tau\P_{(\za,f)}$ is isomorphic to $\P_{(\tau\za,\tau f)}$. Furthermore, the oriented intersection $\aaa$ gives rise to the canonical homomorphism from $\P_{(\za,f)}$ to $\tau\P_{(\za,f)}[1]$ that appears in the associated Auslander--Reiten triangle. 
\end{proposition}

\section{Main results}\label{Proofs of Theorems 1 and 2}

In this section, we give proofs of the main results. We assume that $A$ is a graded gentle algebra, and $(\cals,\calm,\zD^*,G)$ is its geometric model. We start by showing some key technical results. 

\subsection{Key technical results}

Let $\ell$ be a $\gpoint$-arc in $\zD$ and $0$ be the trivial grading on $\ell$: its value at the only intersection $\ell^*=\zD^*\cap \ell$ is $0$. Denote by $\P_{(\ell,0)}$ the corresponding object in $\dba$. 
We call $\P_{(\ell,0)}$ \emph{$\tau$-periodic}, if there are integers $m>0$ and $r$ such that $\tau^m\P_{(\ell,0)}\cong\P_{(\ell,0)}[r]$. Note that $\P_{(\ell,0)}$ is $\tau$-periodic exactly if $\P_{(\ell,0)}$ is a fractionally Calabi--Yau object, since $\SSS = \tau[1]$.
The following has been shown in \cite[Proposition 2.16]{O19}, for completeness we include a short proof here.

\begin{lemma}\label{lem:key2}
The object $\P_{(\ell,0)}$ is $\tau$-periodic if $\ell$ is isotopic to a boundary segment of~$\Sigma$. Furthermore, if $\Sigma$ is not a disc or a disc with one $\rpoint$-point in the interior, then for each boundary component $\partial_i$, there exists at least one $\gpoint$-arc $\ell$ in $\zD$ with at least one endpoint on $\partial_i$, and such that $\ell$ is not isotopic to a boundary segment.
\end{lemma}
\begin{proof}
The first statement follows from Proposition \ref{prop:tau}. If $\ell$ is isotopic to a boundary segment of the boundary component $\partial_i$, then  $\tau^{m_i}\ell\sim \ell$, where $m_i$ is the number of  $\gpoint$-points on $\partial_i$,  and $\tau^{m_i}\P_{(\ell,0)}\cong\P_{(\ell,0)}[r]$ for some integer $r$. Thus $\P_{(\ell,0)}$ is $\tau$-periodic. 
For the second statement, note that any $\gpoint$-arc is isotopic to a chain of several $\gpoint$-arcs from $\Delta$. Since $\Sigma$ is connected, any two boundary components can be connected by a $\gpoint$-arc and hence by a chain of 
$\gpoint$-arcs from $\Delta$. Therefore, if  $\Sigma$ has at least two boundary components, then any boundary component can be connected by a $\gpoint$-arc from $\Delta$ to some other boundary component. This arc is obviously not isotopic to a part of the boundary. Otherwise, assume that $\partial$ is connected and all $\gpoint$-arcs in $\Delta$ are isotopic to a part of $\partial$. Then all  $\gpoint$-arcs are isotopic to a part of $\partial$, it follows that the fundamental group of $\Sigma$ (with $\rpoint$-points removed) is generated by $\partial$ and hence is cyclic. It follows that $\Sigma$ is a disc or a disc with one $\rpoint$-point in the interior. 
\end{proof}

\begin{remark}
    We do not explain here that the object $\P_{(\ell,0)}$ is $\tau$-periodic \textbf{only if} $\ell$ is isotopic to a boundary segment of $\Sigma$, but this will automatically follow from the  considerations below.
\end{remark}


Let $\ell$ be a $\gpoint$-arc in $\zD$ with endpoints $p_i$ and $p_j$ on two boundary components $\partial_i$ and $\partial_j$, respectively. For $s_i, s_j>0$, denote by $^{[s_i]}\ell^{[s_j]}$ the $\gpoint$-arc obtained from $\ell$ by a rotation of the endpoints of $\ell$ given by wrapping $s_i$ times around $\partial_i$ and~$s_j$ times around $\partial_j$, respectively, in the negative direction.

First, assume that $i\ne j$, that is, the endpoints of $\ell$ belong to \textbf{distinct} components of the boundary.
There are $s_i+s_j+1$ intersection points $q_{-s_i}$, $q_{-s_i+1}$ $\cdots$, $q_{-1}$, $q_0$, $q_1$, $\cdots$, $q_{s_j-1}$, $q_{s_j}$ between $\ell$ and $^{[s_i]}\ell^{[s_j]}$,
see Figure \ref{fig:silsj}.
\begin{figure}[ht]
	\centering 
		\begin{tikzpicture}[>=stealth,scale=0.8]
			
			\draw[thick,black, fill=gray!40, line width=1.pt] (-6,0) circle (2);
		    \draw[thick,black, fill=gray!40, line width=1.pt] (6,0) circle (2);
			
			\draw [line width=1.5pt, dark-green] (-4,0) to (4,0);

			\draw [line width=1.pt, red!60] (-4.5,-1.35) to (-3,-3.4);
			\draw [line width=1.pt, red!60] (4.5,-1.35) to (3,-3.4);
			\draw [line width=1.pt, red!60, dotted] (0,-3.9) to (-3,-3.4);
			\draw [line width=1.pt, red!60, dotted] (0,-3.9) to (3,-3.4);
			
			\draw [line width=1.pt, red!60] (-4.5,1.35) to (-3,3.4);
			\draw [line width=1.pt, red!60] (4.5,1.35) to (3,3.4);
			\draw [line width=1.5pt, red!60, dotted] (0,3.9) to (-3,3.4);
			\draw [line width=1.pt, red!60, dotted] (0,3.9) to (3,3.4);
			
			\draw [line width=1.pt, red!60] (0,3.9) to (0,-3.9);

			\draw [bend right, line width=1.pt, ->] (-2.4,0) to (-2.85,0.6);
			\draw [bend right, line width=1.pt, ->] (2.9,0) to (2.5,0.6);
			\draw [bend right, line width=1.pt, ->] (-0.6,0) to (-1.3,0.6);

			\draw[cyan,line width=1.5pt,black] (-4,0) to[out=-80,in=0](-6,-2.4)
			to[out=180,in=-90](-8.4,0)
			to[out=90,in=180](-6,2.4)
			to[out=0,in=90](-3.6,0)
			to[out=-90,in=0](-6,-2.8)
			to[out=180,in=-90](-8.8,0)
			to[out=90,in=180](-6,2.8)
			to[out=0,in=90](-3.2,0)
			to[out=-90,in=0](-6,-3.2)
			to[out=180,in=-90](-9.2,0)
			to[out=90,in=180](-6,3.2)
			to[out=0,in=90](-2.8,0)
			to[out=-90,in=0](-6,-3.6)
			to[out=180,in=-90](-9.6,0)
			to[out=90,in=180](-6,3.6)
			to[out=0,in=120](-1,0)
			;
			
			\draw[cyan,line width=1.5pt,black] (4,0) to[out=120,in=180](6,2.4)
			to[out=0,in=90](8.4,0)
			to[out=-90,in=0](6,-2.4)
			to[out=180,in=-90](3.2,0)
			to[out=90,in=180](6,2.8)
			to[out=0,in=90](8.8,0)
			to[out=-90,in=0](6,-2.8)
			to[out=180,in=-90](2.4,0)
			to[out=90,in=180](6,3.2)
			to[out=0,in=90](9.2,0)
			to[out=-90,in=0](6,-3.3)
			to[out=180,in=-30](1,-2)
			to[out=150,in=-60](-1,0)
			;

			\draw[thick,black, fill=white] (8,0) circle (0.1);
			\draw[thick,black, fill=white] (4,0) circle (0.1);
			\draw[thick,black, fill=white] (-4,0) circle (0.1);
			\draw[thick,black, fill=white] (-6,2) circle (0.1);
			\draw[thick,black, fill=white] (6,-2) circle (0.1);
			\draw[thick, fill=red ] (7.5,-1.35) circle (0.1);
			\draw[thick, fill=red ] (4.5,-1.35) circle (0.1);
		    \draw[thick, fill=red ] (-4.5,-1.35) circle (0.1);
			\draw[thick, fill=red ] (-4.5,1.35) circle (0.1);
			\draw[thick, fill=red] (4.5,1.35) circle (0.1);

			\draw (-3.4,3.5) node[red!60] {$\ell^{*}_{i}$};
            \draw (0.5,2) node[red!60] {$\ell^{*}$};
			\draw (-2,2.3) node[red!60] {$\mathbb{P}_{i}$};
			
			\draw (1.5,0.5) node[dark-green] {$\ell$};
	
			\draw (1.7,-1.7) node {$^{[s_{i}]}\ell^{[s_{j}]}$};
			\draw (-7.5,0) node {$\partial_{i}$};
            \draw (7.5,0) node {$\partial_{j}$};
			\draw (-5,0) node[font=\scriptsize] {$p_i{=}q_{-s_{i}}$};
		    \draw (4.8,0) node[font=\scriptsize] {$q_{s_{j}}{=}p_j$};
			\draw (-2.8,-0.3) node[fill=white, inner sep=1pt, font=\scriptsize] {$q_{-1}$};
			\draw (-1.2,-0.3) node[fill=white, inner sep=1pt, font=\scriptsize] {$q_{0}$};
			\draw (2.4,-0.3) node[fill=white, inner sep=1pt, font=\scriptsize] {$q_{1}$};
            \draw (3.2,-0.3) node[fill=white, inner sep=1pt, font=\scriptsize] {$q_{2}$};
		
			\draw (-2.1,0.5) node[font=\small] {$\aaa_{-1}$};
			\draw (-0.5,0.5) node[font=\small] {$\aaa_{0}$};
			\draw (3,0.6) node[font=\small] {$\aaa_{1}$};
			
		\end{tikzpicture}
	\caption{Intersections of $\ell$ and $^{[s_i]}\ell^{[s_j]}$ when $\partial_i$ and $\partial_j$  are distinct (for $s_i=4$ and $s_j=3$)} 
	\label{fig:silsj}
\end{figure}
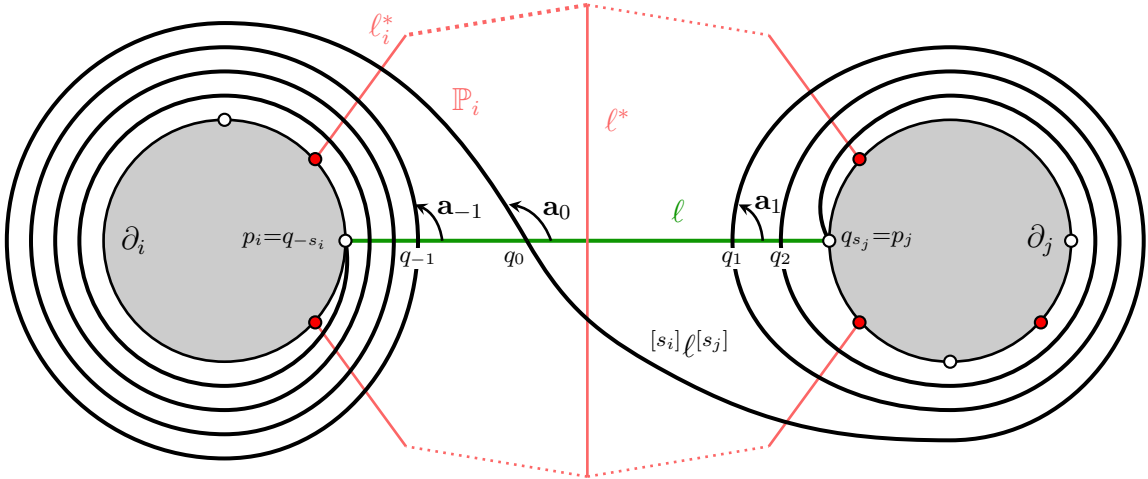

One can see that there are no other intersections between $\ell$ and $^{[s_i]}\ell^{[s_j]}$. Moreover,  $\ell$ and $^{[s_i]}\ell^{[s_j]}$ are in minimal position. Indeed, note that all intersections $q_k$, $-s_i\le k\le s_j$, have the same orientation (namely, negative as intersections from $\ell$ to $^{[s_i]}\ell^{[s_j]}$) and hence cannot be removed by isotopy (recall that $\Sigma$ is oriented). It follows that  the corresponding object $\P_{(\ell,0)}$ is not $\tau$-periodic.

Now assume $i=j$, that is, both endpoints of $\ell$ are on \textbf{the same}  boundary component~$\partial_i$. Assume additionally  that $\ell$ is not isotopic to a part of $\partial_i$. In this case we have $s_i=s_j$.
There are $2s_i+1$ intersection points $q_{-s_i}$,  $\cdots$, $q_0$, $\cdots$,  $q_{s_i}$ between $\ell$ and $^{[s_i]}\ell^{[s_i]}$ as in the case $i\ne j$. Additionally, there are ``extra'' intersections that we label $\bar q_{-s_i}, \cdots, \bar q_{-1}, \bar q_1, \cdots, \bar q_{s_i}$,
see Figure \ref{fig:silsi}. We claim that the arcs $\ell$ and $^{[s_i]}\ell^{[s_i]}$ are in minimal position. Indeed, a pair of intersection points can be removed by isotopy only if the points in the pair are adjacent on both $\ell$ and $^{[s_i]}\ell^{[s_i]}$. One can see that there are two such pairs: $q_0,\bar q_{-1}$ and $q_0,\bar q_{1}$. But an isotopy removing either of these two pairs is equivalent to an isotopy contracting $\ell$ to a part of $\partial_i$, which is impossible by our assumption. Therefore,  the object $\P_{(\ell,0)}$ is also not $\tau$-periodic in this case.

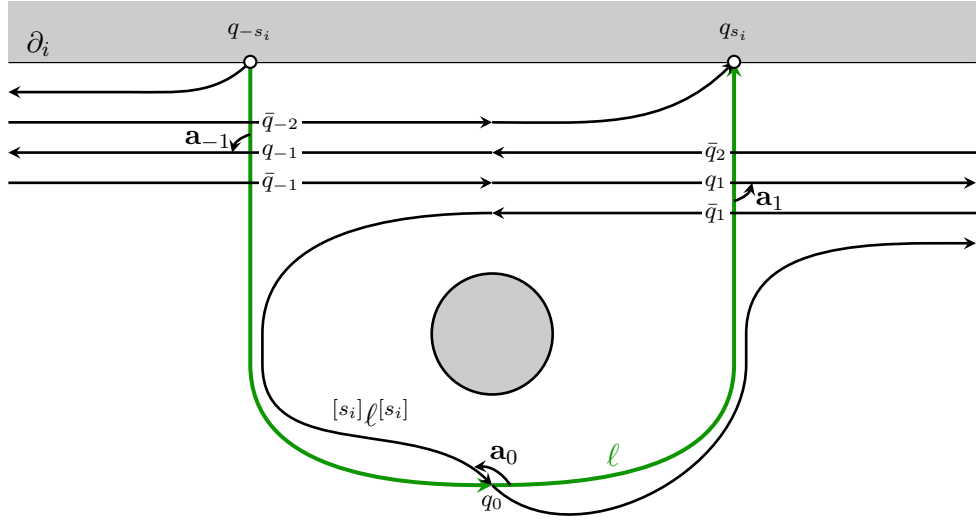
\begin{figure}[ht]
	\centering 
		\begin{tikzpicture}[>=stealth,scale=0.8]

            \draw [line width=1pt] (-8,3) to (8,3);	
            \fill[gray!40, draw=none] (-8,3) rectangle (8,4); 

            \draw[black, fill=gray!40, line width=1pt] (0,-1.5) circle (1);   

            \draw[line width=1.5pt,dark-green, ->] (-4,3) to (-4,-2) 
			to[out=-90,in=-180](0,-4);
		    \draw[line width=1.5pt,dark-green, ->] (0,-4) to[out=0,in=-90] (4,-2)
			to (4,3);

            \draw[line width=1pt,->] (-4,3) to[out=-135,in=0] (-6,2.5) to (-8,2.5);
			\draw[line width=1pt,->] (-8,2) to (0,2);
            \draw[line width=1pt,->] (0,2) to[out=0,in=-135] (4,3);
            \draw[line width=1pt,->] (8,1.5) to (0,1.5);
            \draw[line width=1pt,->] (0,1.5) to (-8,1.5);
            \draw[line width=1pt,->] (-8,1) to (0,1);
            \draw[line width=1pt,->] (0,1) to (8,1);
            \draw[line width=1pt,->] (8,0.5) to (0,0.5);
            \draw[line width=1pt,->] (0,0.5) to[out=-180,in=90]  (-3.8,-1.5)
                 to (-3.8,-2) to[out=-90, in=135]  (0,-4);
            \draw[line width=1pt,->] (0,-4) to[out=-45,in=-90]  (4.2,-2) to (4.2,-1.5) to[out=90,in=-180] (8,0);

			\draw [bend right, line width=1pt, ->] (0.3,-4) to (-0.3,-3.7);   
			\draw [bend right, line width=1pt, ->] (4,0.7) to (4.3,1);
			\draw [bend right, line width=1pt, ->] (-4,1.8) to (-4.3,1.5);

			\draw[thick,black, fill=white] (-4,3) circle (0.1);  
			\draw[thick,black, fill=white] (4,3) circle (0.1);

			\draw (2,-3.5) node[dark-green] {$\ell$};
			\draw (-2,-2.8) node {$^{[s_{i}]}\ell^{[s_{i}]}$};
			\draw (-7.5,3.3) node[font=\small] {$\partial_{i}$};
			\draw (-4,3.5) node[font=\scriptsize] {$q_{-s_{i}}$};
		    \draw (4,3.5) node[font=\scriptsize] {$q_{s_{i}}$};
			\draw (-3.5,1.5) node[fill=white, inner sep=1pt, font=\scriptsize] {$q_{-1}$};
            \draw (-3.5,1) node[fill=white, inner sep=1pt, font=\scriptsize] {$\bar q_{-1}$};
            \draw (-3.5,2) node[fill=white, inner sep=1pt, font=\scriptsize] {$\bar q_{-2}$};
			\draw (0,-4.3) node[font=\scriptsize] {$q_{0}$};
			\draw (3.7,1) node[fill=white, inner sep=1pt, font=\scriptsize] {$q_{1}$};
            \draw (3.7,0.5) node[fill=white, inner sep=1pt, font=\scriptsize] {$\bar q_{1}$};
            \draw (3.7,1.5) node[fill=white, inner sep=1pt, font=\scriptsize] {$\bar q_{2}$};
			\draw (-4.7,1.7) node[font=\small] {$\aaa_{-1}$};
			\draw (0.2,-3.5) node[font=\small] {$\aaa_{0}$};
			\draw (4.6,0.7) node[font=\small] {$\aaa_{1}$};
		\end{tikzpicture}
	\caption{Intersections of $\ell$ and $^{[s_i]}\ell^{[s_i]}$ when both endpoints of $\ell$ are on the same boundary component $\partial_i$ (for $s_i=2$)} 
	\label{fig:silsi}
\end{figure}

$\cdots$, $q_{s_j-1}$, $q_{s_j}$, respectively.
Denote by $\aaa_{-s_i+1}$, $\aaa_{-s_i+2}$ $\cdots$, $\aaa_{-1}$, $\aaa_0$, $\aaa_1$, $\cdots$, $\aaa_{s_j-2}$, $\aaa_{s_j-1}$ the oriented intersections from $\ell$ to $^{[s_i]}\ell^{[s_j]}$ associated with intersections $ q_{-s_i+1}$, $q_{-s_i+2}$ $\cdots$, $q_0$, $q_1$, $\cdots$, $q_{s_j-2}$, $q_{s_j-1}$, respectively. Since $\ell$ and 
$^{[s_i]}\ell^{[s_j]}$ are in minimal position, these points contribute to the graded $\Hom$ space between the objects
in $\dba$ corresponding to  $\ell$ and $^{[s_i]}\ell^{[s_j]}$ (with some gradings) by Proposition~\ref{prop:obj-in-derived-cat}.

For a boundary component $\partial_i$ of $\Sigma$, the authors of \cite{LP20} introduce a combinatorial boundary component $B_i$, which is a composition of so-called forbidden and permitted paths between arcs in $\zD$,  giving rise to  a simple closed curve surrounding $\partial_i$. They then define the combinatorial winding number of $B_i$ and prove that it is  equal to the geometric winding number of the grading line field $\eta$ around $\partial_i$, see \cite[Proposition 3.10]{LP20}. Similarly, we can define a dual version of the combinatorial boundary component associated to $\partial_i$, given by a composition of certain paths between arcs in $\zD^*$, as well as an associated combinatorial winding number. Analogously to the proof of \cite[Proposition 3.10]{LP20} this dual combinatorial winding number is again equal to the geometric winding number of $\eta$ around $\partial_i$. 
Explicitly, these combinatorial winding numbers
are computed as follows.

Let $\partial_1,\ldots,\partial_b$ be the connected components of $\partial$. Let $m_i$ be the number of marked points on $\partial_i$, let $l_i$ be the number of arcs in $\Delta^*$ attached to $\partial_i$. Note that $l_i$ equals the number of 
polygons of $\Delta^*$-dissection contacting $\partial_i$, see Figure~\ref{fig:winding}. Note that $m_i$ of these polygons contact $\partial_i$ by a segment between two  $\rpoint$-points and $u_i:=l_i-m_i$ of them contact $\partial_i$ at a single $\rpoint$-point.
For any $j\in\{1\ldots m_i\}$ consider the corresponding segment of $\partial_i$ with a marked point and the adjacent polygon $\mathbb P_j$. 
Let $\omega_{i,j}^m$ be the sum of degrees of all minimal oriented intersections formed by the edges of $\mathbb P_j$.
For any $j\in\{1\ldots u_i\}$ let  $\omega_{i,j}^u$ be the degree of the corresponding minimal oriented intersection of $\Delta^*$-arcs at the contact point. Now define
\begin{equation}
    \label{eq_w}
    \omega_i:=\sum_{j=1}^{m_i} (1-\omega^m_{i,j})-\sum_{j=1}^{u_i}(1-\omega^u_{i,j}), 
\end{equation}
Note that in the ungraded case one has $\omega_i=m_i-u_i$.

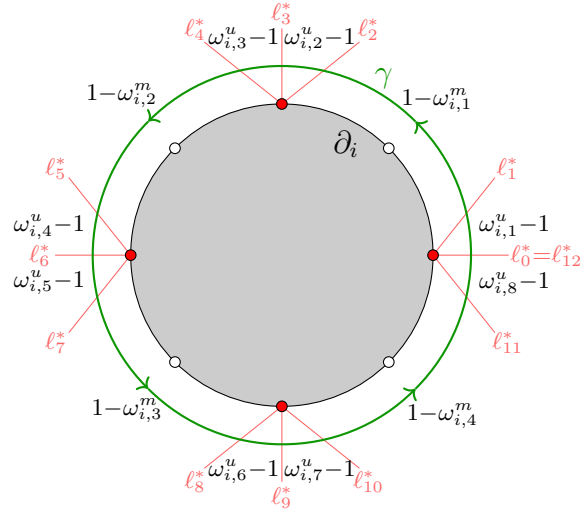
\begin{figure}[ht]
		\begin{tikzpicture} 
                \draw[fill=gray!40] (0,0) circle (2);  

				\foreach \u in {1,...,4} 
    				\draw[red!60]  ({2*cos(90*\u)},{2*sin(90*\u)}) to ({3*cos(90*\u)},{3*sin(90*\u)});
                \foreach \u in {1,...,4} 
                    \draw[red!60]  ({2*cos(90*\u)},{2*sin(90*\u)}) to ({3*cos(20+90*\u)},{3*sin(20+90*\u)});
                \foreach \u in {1,...,4} 
                    \draw[red!60]  ({2*cos(90*\u)},{2*sin(90*\u)}) to ({3*cos(-20+90*\u)},{3*sin(-20+90*\u)});

                \foreach \u in {1,...,4} 
                    \draw[thick, dark-green, ->]  ({2.5*cos(45+90*\u)},{2.5*sin(45+90*\u)}) arc[start angle={45+90*\u}, end angle=135+90*\u, radius=2.5];

                \foreach \u in {1,...,4} 
  				\draw[fill=red] ({2*cos(90*\u)},{2*sin(90*\u)}) circle (2pt);

                \foreach \u in {1,...,4} 
  				\draw[fill=white] ({2*cos(45+90*\u)},{2*sin(45+90*\u)}) circle (2pt);

                \draw[dark-green] ({2.7*cos(60)}, {2.7*sin(60)}) node {$\gamma$};
                \draw ({1.7*cos(60)}, {1.7*sin(60)}) node {$\partial_i$};
				\draw[font=\scriptsize] ({3.1*cos(7)}, {3.1*sin(7)}) node {$\omega_{i,1}^u{-}1$};
                \draw[font=\scriptsize] ({2.9*cos(80)}, {2.9*sin(80)}) node {$\omega_{i,2}^u{-}1$};
                \draw[font=\scriptsize] ({2.9*cos(100)}, {2.9*sin(100)}) node {$\omega_{i,3}^u{-}1$};
                \draw[font=\scriptsize] ({3.1*cos(173)}, {3.1*sin(173)}) node {$\omega_{i,4}^u{-}1$};
                \draw[font=\scriptsize] ({3.1*cos(187)}, {2.9*sin(187)}) node {$\omega_{i,5}^u{-}1$};
                \draw[font=\scriptsize] ({2.9*cos(260)}, {2.9*sin(260)}) node {$\omega_{i,6}^u{-}1$};
                \draw[font=\scriptsize] ({2.9*cos(280)}, {2.9*sin(280)}) node {$\omega_{i,7}^u{-}1$};
                \draw[font=\scriptsize] ({3.1*cos(-7)}, {3.1*sin(-7)}) node {$\omega_{i,8}^u{-}1$};

                \draw[font=\scriptsize] ({2.9*cos(45)}, {2.9*sin(45)}) node {$1{-}\omega_{i,1}^m$};
                \draw[font=\scriptsize] ({3*cos(135)}, {3*sin(135)}) node {$1{-}\omega_{i,2}^m$};
                \draw[font=\scriptsize] ({2.9*cos(-135)}, {2.9*sin(-135)}) node {$1{-}\omega_{i,3}^m$};
                \draw[font=\scriptsize] ({3*cos(-45)}, {3*sin(-45)}) node {$1{-}\omega_{i,4}^m$};

                \foreach \u [evaluate=\u as \var using int(3*\u)] in {1,...,3} 
                    \draw[red!60, font=\scriptsize]  ({3.2*cos(90*\u)},{3.2*sin(90*\u)}) node {$\ell^*_{\var}$};
                \foreach \u [evaluate=\u as \var using int(3*\u+1)] in {0,...,3} 
                    \draw[red!60, font=\scriptsize]  ({3.2*cos(90*\u+21)},{3.2*sin(90*\u+21)}) node {$\ell^*_{\var}$};
                \foreach \u [evaluate=\u as \var using int(3*\u-1)] in {1,...,3} 
                    \draw[red!60, font=\scriptsize]  ({3.2*cos(90*\u-21)},{3.2*sin(90*\u-21)}) node {$\ell^*_{\var}$};
                \draw[red!60, font=\scriptsize]  ({3.5*cos(90*0)},{3.5*sin(90*0)}) node {$\ell^*_{0}{=}\ell^*_{12}$};
                \draw[red!60, font=\scriptsize]  ({3.2*cos(90*0-21)},{3.2*sin(90*0-21)}) node {$\ell^*_{11}$};
		  \end{tikzpicture}
	\caption{Grading increment under a winding around $\partial_i$. Here $m_i=4$, $l_i=12$, $u_i=8$.} 
	\label{fig:winding}
\end{figure}

\begin{lemma}
    \label{lemma_oneturn}
With the notations above, let $\ell^*$ be an arc in $\Delta^*$ with an endpoint on $\partial_i$. Let~$\gamma$ be a curve intersecting $\ell^*$ at points $x_0,x_1$ and winding around $\partial_i$ in the negative direction between $x_0$ and $x_1$. Let $\ell^*=\ell^*_0,\ell^*_1,\ldots,\ell^*_{l_i}=\ell^*$ be the $\Delta^*$-arcs that $\gamma$ intersects consecutively between $x_0$ and $x_1$,  see Figure~\ref{fig:winding}. If $f$ is a grading on $\gamma$ then 
\begin{equation}
    \label{eq_oneturn}
    f(\ell^*_{l_i})=f(\ell^*_{0})+\omega_i.
\end{equation}
\end{lemma}
\begin{proof}
This follows from the definition of a grading and the definition of $\omega_i$. More precisely, for any segment of $\gamma$ between between $\ell^*_{k-1}$ and $\ell^*_k$, if it belongs to a polygon that contacts $\partial_i$ at a single $\rpoint$-point, then by~\eqref{equ:grading} one has $f(\ell^*_k)-f(\ell^*_{k-1})=-1+\omega_{i,j}^u$ for some $j\in\{1,\ldots, u_i\}$, where $\omega_{i,j}^u$ is the corresponding degree.

Similarly, we have $f(\ell^*_k)-f(\ell^*_{k-1})=1-\omega_{i,j}^m$ for some $j\in\{1,\ldots, m_i\}$ and the corresponding degree $\omega_{i,j}^m$, if the segment of $\gamma$ between $\ell^*_{k-1}$ and $\ell^*_k$ belongs to a polygon that contacts $\partial_i$ along a segment between two  $\rpoint$-points. Summing up over~$k$ one gets $f(\ell^*_{l_i})-f(\ell^*_0)=\omega_i$ by~\eqref{eq_w}. 
\end{proof}

\begin{lemma}\label{lem:key3}
With the notations above, assume that $\ell$ and $^{[s_i]}\ell^{[s_j]}$ are equipped with some gradings $f$ and $g$. Then the degrees of oriented intersections $\aaa_k$ from $(\ell,f)$ to  $(^{[s_i]}\ell^{[s_j]},g)$ satisfy
$$
\begin{array}{ll}
					|\aaa_{-k}|=|\aaa_0|+k\omega_i, & 0\le k\le s_i-1; \\
					|\aaa_k|=	|\aaa_0|+k\omega_j, & 0\le k\le s_j-1,
                \end{array}
$$
where $\omega_i$ and $\omega_j$ are the winding numbers of the boundary components $\partial_i$ and $\partial_j$, respectively.           
\end{lemma}
\begin{proof}
It follows from Definition~\ref{def_degree-of-orint} of the degree of an oriented intersection. Let $\mathbb P_i$ be the polygon of the dissection $\Delta^*$ where  $\aaa_{-k}$ belong for $0\le k\le s_i-1$, let 
$\ell^*_i$ be the first arc in $\Delta^*$ attached to $\partial_i$ after $p_i$ in the negative direction on $\partial_i$, see Figure~\ref{fig:silsj}. Then $^{[s_i]}\ell^{[s_j]}$ exits $\mathbb P_i$ via $\ell^*_i$ after $\aaa_{-k}$ for $0\le k\le s_i-1$. We are at case 3 from Figure~\ref{Figure2} and take $\ell^*$ as  $\ell^*_1$ and  $\ell^*_i$ as $\ell^*_2$.

Let $_0\ell_i^*,_1\ell_i^*, _2\ell_i^*, \ldots$ be the occurrences of $\ell^*_i$ in the ordered set of crossings $\zD^*\cap ^{[s_i]}\ell^{[s_j]}$, ordered from $q_0$ towards $q_{-s_i}$. Then we have $g(_k\ell_i^*)=g(_0\ell_i^*)+k\omega_i$ by Lemma~\ref{lemma_oneturn}. Then by 
Definition~\ref{def_degree-of-orint} we get
$$|\aaa_{-k}|-|\aaa_0|=(g(_k\ell_i^*)-f(\ell^*)+|\sigma|)-(g(_0\ell_i^*)-f(\ell^*)+|\sigma|)=
g(_k\ell_i^*)-g(_0\ell_i^*)=k\omega_i.$$

For the degree of $|\aaa_k|$, the arguments are the same.
\end{proof}

We now consider the $\tau$-orbits of the indecomposable projective  $A$-modules.
Denote by~$m$ the least common multiple of the $m_1,\ldots,m_b$ and set $s_i=\frac{m}{m_i}$. 
Let $\ell$ be an $\gpoint$-arc in $\zD$. Assume that the endpoints of $\ell$ are $p_i$ and $p_j$ which belong to $\partial_i$ and $\partial_j$, respectively. The object  $\P_{(\ell,0)}\in\dba$ associated to the graded arc $(\ell,0)$ (where the grading is given by the only value at $\ell^*$ being $0$) is the indecomposable projective  $A$-module generated in degree $0$.
By Proposition \ref{prop:tau}, for any positive integer  $N$ the object $\tau^{mN}\P_{(\ell,0)}$ arises from a graded arc $(^{[s_iN]}\ell^{[s_jN]},g_N)$ for the grading $g_N=\tau^{mN}0$.

We have the following key lemma.

\begin{lemma}\label{lem:2.3}
Using the notations above, there exist integers $v,v'$ independent of  $N$, such that for all $N>>0$ 
\begin{align*}\text{the left length of}\quad  \tau^{mN}\P_{(\ell,0)} &=  \min\{s_i\omega_i,s_j\omega_j,0\}\cdot N+v,\\
\text{the right length of}\quad  \tau^{mN}\P_{(\ell,0)} &=  \max\{s_i\omega_i,s_j\omega_j,0\}\cdot N+v'.
\end{align*}
\end{lemma}

\begin{proof}
Recall that the object $\tau^{mN}\P_{(\ell,0)}$ corresponds to the graded arc $(^{[s_iN]}\ell^{[s_jN]},g_N)$ for the grading $g_N=\tau^{mN}0$. By definition, the left length of $\tau^{mN}\P_{(\ell,0)}$ is the minimal value of the grading function $g_N$. Let $\ell^*_1,\ldots,\ell^*_{l_i}$ be the arcs in $\Delta^*$ attached to $\partial_i$. 
For any of them, let $_1\ell^*_k,\ldots,_{s_iN}\ell^*_k$ be the occurrences of $\ell^*_k$ in the ordered set of crossings $\zD^*\cap \tau^{mN}\ell$, counted starting from $q_0$. By Lemma~\ref{lemma_oneturn}, one has 
$g_N(_{n+1}\ell^*_k)=g_N(_{n}\ell^*_k)+\omega_i$, and hence 
$g_N(_{n}\ell^*_k)=g_N(_{1}\ell^*_k)+\omega_i(n-1)$ for all $n$ and $k$. 
If $\omega_i\ge 0$ we have that
the minimum of $\{g_N(_n\ell^*_k)\}_{k=1\ldots l_i, n=1\ldots s_iN}$ is attained at the first lap where $n=1$ and is
$$\min(g_N(_1\ell^*_1),\ldots,g_N(_1\ell^*_{l_i}))=\mathrm{const}.$$
Now if $\omega_i< 0$ then the minimum of $\{g_N(_n\ell^*_k)\}_{k=1\ldots l_i, n=1\ldots s_iN}$ is attained at the last lap where $n=s_iN$ and is
$$\min(g_N(_1\ell^*_1),\ldots,g_N(_1\ell^*_{l_i}))+(s_iN-1)\omega_i=s_i\omega_iN+\mathrm{const}.$$

Note that $g_{N_1}(_{n}\ell^*_k)=g_{N_2}(_{n}\ell^*_k)$ for positive integers $N_1$ and $N_2$, as soon as both intersections exist. Therefore, the above constants are independent of the choice of $N$.

Similarly, we compute the minimum of $g_N$ on the intersections with the arcs from $\Delta^*$ attached to $\partial_j$.
The statement now follows by considering several cases depending on how $s_i\omega_i, s_j\omega_j$, and $0$ compare.

For the right length the arguments are similar.
\end{proof}

\begin{corollary}\label{cor:key1}
In the above notation, there exist  constants $v,v'$ independent on $N$ such that 
$$\Hom(A,\tau^{mN}A[n])=0$$
for each $N>0$ and  
$$n< \min(s_1\omega_1,\ldots,s_b\omega_b,0)\cdot N+v\quad \text{or}\quad n > \max(s_1\omega_1,\ldots,s_b\omega_b,0)\cdot N+v'.$$
\end{corollary}
\begin{proof}
First we check that for any arc $\ell\in\Delta$ whose endpoints are on $\partial_i$ and $\partial_j$ there exists $v_{\ell},v'_{\ell}$, independent on $N$, such that  
$$\Hom(A,\tau^{mN}\P_{(\ell,0)}[n])=0$$ for each $N>0$ and 
$$n<\min\{s_i\omega_i,s_j\omega_j,0\}\cdot N+v_{\ell} \quad \text{or}\quad n>\max\{s_i\omega_i,s_j\omega_j,0\}\cdot N+v'_{\ell}.$$
Indeed, note that $\Hom_{\dba}(A,\tau^{mN}\P_{(\ell,0)}[n])\cong H^n(\tau^{mN}\P_{(\ell,0)})$ and that 
$\tau^{mN}\P_{(\ell,0)}$, as a graded module, is the direct sum~\eqref{eq_P-direct-sum} of indecomposable projective $A$-modules placed in degrees given by the values of $g_N=\tau^{mN}0$. Since any indecomposable projective $A$-module is supported at finitely many degrees, the claim follows from Lemma~\ref{lem:2.3} (recall that the left/right length of $\tau^{mN}\P_{(\ell,0)}$ is the minimal/maximal value of $g_N$).

Now the statement of the corollary follows since $A\cong \oplus_{\ell\in\Delta}\P_{(\ell,0)}$. 
\end{proof}

\begin{remark}
    Note that the graded module $\tau^{mN}\P_{(\ell,0)}$ constructed by winding $\ell$ around boundary components $\partial_i$ and $\partial_j$ close to them may be not in minimal position with the arcs from $\Delta^*$: there may be unnecessary intersections with the arcs that are isotopic to a segment of $\partial_i$ or $\partial_j$. Therefore results from Lemma~\ref{lem:2.3} actually only provide a lower/upper bound on the left/right length of the ``minimal representative'' of $\tau^{mN}\P_{(\ell,0)}$. Nevertheless, the following lemma shows that these  bounds are asymptotically exact if the arc $\ell$ is not isotopic to a boundary segment.
\end{remark}

\begin{lemma}\label{lem:2.8}
In the above notation, suppose that the arc $\ell$ is not isotopic to a boundary segment. 
If $n=1$ or $n=s_i\omega_iN+(1-\omega_i)$  or 
$n=s_j\omega_jN+(1-\omega_j)$ then $$\Hom(A, \tau^{mN}\P_{(\ell,0)}[n])\ne 0.$$ 
\end{lemma}
\begin{proof}
Recall that there are oriented intersections  $\aaa_{k}$  from $\ell$ to $\tau^{mN}\ell$, where $-s_iN+1\le k\le s_jN-1$, and $|\aaa_{k}|$ is given by Lemma~\ref{lem:key3}. Fix the gradings $f=0$ on $\ell$ and $g_N=\tau^{mN}0$ on $\tau^{mN}\ell$. We claim that $|\aaa_0|=1$ for all $N$. Indeed, for any $r>0$ the degree of the intersection from $(\ell,0)$ to $(\tau^r\ell, \tau^r0)$ at $q$ is given by a formula from Definition~\ref{def_degree-of-orint}, and all ingredients of this formula do not depend on~$r$. For $r=1$ the degree is $1$ by definition. Therefore,  $|\aaa_{-s_iN+1}|=1+\omega_i(s_iN-1)$,  $|\aaa_{s_jN-1}|=1+\omega_j(s_jN-1)$ by Lemma~\ref{lem:key3}. Since~$\ell$ is not isotopic to a boundary segment, arcs $\ell$ and $\tau^{mN}\ell$ are in minimal position. We have $\Hom(\P_{(\ell,0)}, \tau^{mN}\P_{(\ell,0)}[n])\ne 0$ by Proposition~\ref{prop:obj-in-derived-cat} for $n=|\aaa_0|,|\aaa_{-s_iN+1}|,|\aaa_{s_jN-1}|$. Now the statement follows since $\P_{(\ell,0)}$ is a direct summand in $A$.
\end{proof}

Furthermore, we have the following
\begin{lemma}\label{lem:2.9}
There is a constant $u$, independent on $N$, such that for all $N>0$ one has 
$$\sum_{n\in\Z}\dim_k\Hom(A,\tau^{mN}A[n])\le u\cdot N.$$ 
\end{lemma}
\begin{proof}
Denote $\Hom^\bullet(-,-):=\oplus_{n\in \Z}\Hom(-,-[n])$.
Note that we have
\begin{equation*}
\begin{array}{lll}
\Hom^\bullet(A,\tau^{mN}A)&\cong&
\Hom^\bullet(\bigoplus\limits_{\ell'\in \zD}\P_{(\ell',0)},\bigoplus\limits_{\ell\in \zD}\tau^{mN}\P_{(\ell,0)})\\
&\cong&
\bigoplus\limits_{\ell',\ell\in \zD}\Hom^\bullet(\P_{(\ell',0)},\P_{(\tau^{mN}\ell,\tau^{mN}0)}).
\end{array}
\end{equation*}
By Proposition~\ref{prop:obj-in-derived-cat}, the dimension of the graded vector space $\Hom^\bullet(\P_{(\ell',0)},\P_{(\tau^{mN}\ell,\tau^{mN}0)})$ equals to the number of oriented intersections from $\ell'$ to $\tau^{mN}\ell$. More precisely, the dimension equals to the number of oriented intersections between arcs in minimal position isotopic to $\ell'$ and $\tau^{mN}\ell$, which number is bounded from above by the number of oriented intersections from $\ell'$ to $\tau^{mN}\ell$.

Recall  that the associated arc of $\tau^{mN}\ell$ is $^{[s_iN]}\ell^{[s_jN]}$. 
Then the number of intersections between $\ell'$ and $^{[s_iN]}\ell^{[s_jN]}$ is bounded from above by $2+2s_iN+2s_jN$ (where $2$ is the maximal number of intersections between $\ell'$ and $\ell$, and the summands $2s_iN,2s_jN$ are upper bounds on the numbers of intersections that occur near $\partial_i$ and $\partial_j$ respectively). Furthermore, since there are only finitely many  pairs  $(\ell',\ell)$, there exists a positive integer~$u$ such that $\dim_k \Hom^\bullet(A,\tau^{mN}A)\leq uN$ for each $N>0$.
\end{proof}

\subsection{Proofs of the main results}

In this section, we prove our main results. We begin by proving Theorem~\ref{thm:main1gentle}, which is more general than Theorem~\ref{thm:main1}, in that we also allow graded gentle algebras which are not homologically smooth. Theorem~\ref{thm:main1} then is a direct corollary of Theorem~\ref{thm:main1gentle}. Indeed, in the homologically smooth case, we have $H_t =h_t$ by Theorem~\ref{thm:dhkk1}. Noting that  $\mS = \tau[1]$, we have  by Lemma~\ref{lem:pre} (2) that $H_t(\mathbb{S})=H_t(\tau)+t$, so the claim of Theorem~\ref{thm:main1} follows.

\begin{theorem}\label{thm:main1gentle}
    Let $A$ be a proper graded gentle algebra such that its geometric model 
    is not a disc with $\leq 1$ stops in the interior.
    Then  the function $h_t(\tau)$, for $\tau$ the Auslander--Reiten translate in $\dba$ is  given by
$$h_t(\tau) = \left\{\begin{array}{ll}
						- (\min  \Omega) t & t\geq 0; \\ & \\
						-(\max \Omega) t & t\leq 0.
                        \end{array}\right.$$
\end{theorem}

\begin{proof}
By Lemma \ref{lem:key2}, there exists, for any boundary component $\partial_i$, an arc in $\Delta$ with one of the ends on $\partial_i$ and which is not isotopic to a part of $\partial_i$. 

By Theorem \ref{thm:dhkk1} replacing $N$ by $mN$, we have 
\begin{equation}\label{eq:entropyA}
h_t(\tau)=\lim\limits_{N\to\infty}\frac{1}{mN}\log\sum_{n\in\mathbb{Z}}\ \dim_k\Hom(A,\tau^{mN}(A)[n])\cdot e^{-nt}.\end{equation}

Assume first that $t\ge 0$.

\textbf{Case 1:} $\min(\omega_i)<0$. Without loss of generality let us assume that $\min\Omega= \frac{\omega_1}{m_1}$ and 
$$\min\{s_1\omega_1,\ldots,s_b\omega_b,0\}=\min\left\{\frac{m_1}m\omega_1,\ldots,\frac{m_b}m\omega_b\right\}=\frac{m_1}m\omega_1=s_1\omega_1.$$
In this case we have (where $v$ and $u$ are some constants, the first equality is by Corollary~\ref{cor:key1} and the last inequality is by Lemma~\ref{lem:2.9})
\begin{multline}
\label{eq_upperbound}
\sum_{n\in\mathbb{Z}} \dim_k\Hom(A,\tau^{mN}(A)[n])\cdot e^{-nt} =
\sum_{n\ge s_1\omega_1N+v} \dim_k\Hom(A,\tau^{mN}(A)[n])\cdot e^{-nt} \\
\le e^{-t(s_1\omega_1N+v)} \cdot \sum_{n\ge s_1\omega_1N+v} \dim_k\Hom(A,\tau^{mN}(A)[n]) \le e^{-t(s_1\omega_1N+v)} \cdot uN.
\end{multline}
Now we obtain a lower bound. Choose an arc $\ell\in\Delta$ with an end on $\partial_1$ and not isotopic to a boundary segment. By Lemma~\ref{lem:2.8} one has $\Hom(A,\tau^{mN}\P_{(\ell,0)}[n])\ne 0$ where $n=s_1\omega_1N+(1-\omega_1)$. 
It follows that (since $\P_{(\ell,0)}$ is a direct summand in $A$)
\begin{equation}
\label{eq_lowerbound}
    \sum_{n\in\mathbb{Z}} \dim_k\Hom(A,\tau^{mN}(A)[n])\cdot e^{-nt} \ge e^{-t(s_1\omega_1N+(1-\omega_1))}.
\end{equation}
Plugging~\eqref{eq_upperbound} and~\eqref{eq_lowerbound} into~\eqref{eq:entropyA} and computing the limits one gets bounds
$$-t\frac{s_1\omega_1}m\le h_t(\tau)\le -t\frac{s_1\omega_1}m.$$
Recall that $-t \frac{s_1\omega_1}m=-t\frac{\omega_1}{m_1} =-t\min\Omega$ to conclude.

\textbf{Case 2:} $\omega_i\ge 0$ for all $i$. It means that $\min\Omega=0$ and we have to prove $h_t(\tau)=0$, the arguments are similar to the above.

For the upper bound we have (where $v$ and $u$ are some constants, the first equality is by Corollary~\ref{cor:key1} and the last inequality is by Lemma~\ref{lem:2.9})
\begin{multline}
\label{eq_upperbound0}
\sum_{n\in\mathbb{Z}} \dim_k\Hom(A,\tau^{mN}(A)[n])\cdot e^{-nt} =
\sum_{n\ge v} \dim_k\Hom(A,\tau^{mN}(A)[n])\cdot e^{-nt} \le \\
\le e^{-tv} \cdot \sum_{n\ge v} \dim_k\Hom(A,\tau^{mN}(A)[n]) \le e^{-tv} \cdot uN.
\end{multline}
For the lower bound, choose any arc $\ell\in\Delta$ not isotopic to a boundary segment. By Lemma~\ref{lem:2.8} one has $\Hom(A,\tau^{mN}\P_{(\ell,0)}[1])\ne 0$.
It follows that 
\begin{equation}
\label{eq_lowerbound0}
    \sum_{n\in\mathbb{Z}} \dim_k\Hom(A,\tau^{mN}(A)[n])\cdot e^{-nt} \ge e^{-t}.
\end{equation}
As before, plugging~\eqref{eq_upperbound} and~\eqref{eq_lowerbound} into~\eqref{eq:entropyA} and computing the limits one gets bounds
$$0\le h_t(\tau)\le 0.$$
Hence $h_t(\tau)=0$ and we are done.

For $t\le 0$ the arguments are similar.
\end{proof}

Next, Theorem \ref{thm:main2} is a direct consequence of Theorem \ref{thm:main1gentle} and \cite[Proposition 6.13]{EL21} or \cite[Theorem 1.3]{FF23}. For completeness, we repeat the statement and include a short proof. 

\begin{theorem}[Theorem \ref{thm:main2}]\label{thm:Main2Section2} 
Let $\calt=\mathcal{W}(\Sigma, M, \eta)$ be the partially wrapped Fukaya category of $(\Sigma,M,\eta)$.  Suppose further that  $\Sigma$ is not a disc with $\le 1$ stops in the interior.
Then    
$$\uSdim \calt = (1-\min \Omega) \;\; \mbox{ and }\;\;  \lSdim \calt=(1-\max \Omega). $$
\end{theorem}

\begin{proof}
In \cite[Proposition 6.13]{EL21} it is shown that  for the perfect derived category of a proper dg algebra $A$, 
the  upper Serre dimension $\uSdim$ and the  lower Serre dimension $\lSdim$ of $\rm{per}(A)$ are given by
\begin{equation*}
\uSdim=\lim_{t\to+\infty}{\frac{h_t(\mathbb{S})}{t}} \;\;\;\; \mbox{ and } \;\;\;\; \lSdim=\lim_{t\to-\infty}{\frac{h_t(\mathbb{S})}{t}}
\end{equation*}
(the statement in~\cite{EL21} is made only for smooth dg algebras and for $H_t$, but it is $h_t$ that they work with and the smoothness condition is only needed to replace $h_t$ with~$H_t$.)
Therefore, the claim follows from Theorem~\ref{thm:main1gentle} directly.
\end{proof}

We now give a proof of Theorem \ref{thm:main3} which we  restate here.

\begin{theorem}[Theorem~{\ref{thm:main3}}]\label{thm:Main3Section2}
Suppose $\calt$ is a homologically smooth partially wrapped Fukaya category of a surface $\Sigma$ with stops which is not a disc (or $\calt =\dba$ for a homologically smooth  graded gentle algebra $A$ which is not derived equivalent to a hereditary algebra of type $\mathbb{A}$). Let  $H_t(\mathbb{S})$ be the entropy of the Serre functor $\mathbb{S}$  of $\calt$. Then the following statements are equivalent:

\begin{enumerate}
    \item  $H_t(\mathbb{S})$ is linear;
    \item[(1')] $\lSdim \calt=\uSdim\calt$;
    \item $H_t(\mathbb{S})=t$;
    \item[(2')]  $\lSdim \calt=\uSdim\calt=1$;
    \item $\omega_i=0$ for all boundary components of $\Sigma$;
    \item $m_i=n_i$ for all pairs $(m_i,n_i)$ in the AG-invariant;
    \item $\Sigma$ is an annulus and the winding number along the equator is zero; 
    \item $A$ is derived equivalent to a hereditary algebra of affine type $\Tilde{\mathbb A}$ with trivial grading;
    \item $A$ is derived equivalent to a $d$-representation infinite algebra  in the sense of \cite{HIO14}. 
\end{enumerate}
\end{theorem}

\begin{proof}
Equivalences (1) $\Longleftrightarrow$ (1') and (2) $\Longleftrightarrow$ (2') obviously follow from Theorems~\ref{thm:main1} and~\ref{thm:main2}.  Then, by Theorem \ref{thm:main1}, $H_t(\mathbb{S})$ is linear if and only if  $\min  \Omega=\max  \Omega$,  which is the case precisely if   $\Omega=\{0\}$ and if $\min  \Omega=\max  \Omega=0$. So (1) $\Longleftrightarrow$ (2) follows. 
In turn, this is equivalent to (3):
$\omega_i=0$ for each boundary component $\partial_i$. 

Recall that the AG invariant is the set of pairs of numbers $(m_i, m_i - \omega_i$), for $1 \leq i \leq b$, so 
(3) $\Longleftrightarrow$ (4) holds.

Assume that for a surface, the winding number of each boundary component is zero.  
By the Poincar\'e--Hopf index theorem, 
\[\sum\limits_{1\leq i \leq b}(\omega_i+2)=4-4g,\]
where $g$ is the genus of the surface, and the summation extends over all the boundary components,  we have $4g=4-2b$. Since $g\geq 0$ and $b\geq 1$, we deduce that   $g=0$ and $b=2$: $\Sigma$ is a sphere with two holes, that is, an annulus. Such surfaces, with the line field parallel to the boundary, are exactly the geometrical models of hereditary algebras of affine type $\Tilde{\mathbb{A}}$ with trivial grading.
This proves that (3) is equivalent to~(5) and to~(6).

It is shown in \cite[Theorem 7]{H24} for a $d$-representation infinite algebra $A$ with Serre functor $\SSS$ and Coxeter transformation $\Psi$ that the entropy is given by $H_t(\SSS)=dt + \log \rho (\Psi)$ where $\rho(\Psi)$ is the spectral radius of $\Psi$. Hence (7) implies (1). 
Conversely, an algebra of affine type $\Tilde{\mathbb{A}}$ is $1$-representation infinite, therefore (6) implies  (7).
\end{proof}

\section{Entropy of the Serre functor versus entropy of the Coxeter transformation}

In this section, for a gentle algebra with trivial grading,  we relate the categorical entropy of the Serre functor to the natural logarithm of the spectral radius of the morphism on the  Grothendieck group of $\dba$ induced by the Serre functor. In analogy to the definitions in \cite{G87,G03,Y87} (see also in \cite{Og14}), we could refer to the latter as the topological entropy of the Serre functor.   This also relates to a conjecture by Kikuta and Takahashi in the case of autoequivalences in the bounded derived category of coherent sheaves of a smooth projective variety \cite[Conjecture 5.3]{KT19}.

Throughout this section, assume that $A = kQ/I$ is a finite dimensional gentle algebra. Then, since $A$ is Gorenstein \cite{GeissReiten},  $\dba$ has  an Auslander--Reiten translation $\tau$ and it induces a morphism $\Psi = [\tau]$ on the Grothendieck group $\calk_0(\dba)$ which is called the \emph{Coxeter transformation of $A$}.
The characteristic polynomial $\psi(t) \in  \mathbb{Z}[t]$ of $\Psi$ is referred to as the \emph{Coxeter polynomial} \cite{ARS95} and it is a derived invariant of $A$, see, for example, \cite{GM24} and the references within. 

We call the natural logarithm of the spectral radius of $\Psi$, $\log \rho (\Psi)$, the \emph{entropy of the Coxeter transformation}. In particular, we have 
$\log \rho (\Psi) = \log \rho ([\SSS])$, where $[\SSS]$ is the endomorphism of $\calk_0(\dba)$ induced by the Serre functor $\SSS$.

\begin{theorem}(Theorem \ref{thm:main4})\label{thm:categorical vs topological entropy}
Let $A$ be a finite dimensional gentle algebra  with Serre functor $\SSS$ in $\dba$ and let $[\SSS]: \calk_0(\dba) \to \calk_0(\dba)$ be the induced map on the  Grothendieck group of $\dba$. Then 
$$h_{0}(\SSS) = \log \rho([\SSS]).$$
Furthermore, if the global dimension of $A$ is finite then $H_{\rm{cat}}(\SSS) = \log \rho([\SSS])$
where $H_{\rm{cat}}(\SSS):= H_0(\SSS)$.
\end{theorem}

\begin{proof}
It follows from Theorem \ref{thm:main1} that $h_0(\SSS)=0$. 
 By \cite[Theorem 6.17]{GM24}, the 
characteristic polynomial of the map $[\SSS]: \calk_0(\dba) \to \calk_0(\dba)$ is a 
cyclotomic polynomial whose roots are all roots of unity.
Therefore the natural logarithm of the spectral radius $\rho([\SSS])$ of $[\SSS]$ is equal to zero. Thus we have $\log \rho([\SSS])=0$, which completes the proof.
\end{proof}

We end this section with a reformulation in terms of surface data of the Coxeter polynomial given in terms of the AG-invariant in  \cite[Theorem 6.17]{GM24}.

\begin{proposition}\label{prop:coxeterpoly} Let $A$ be a finite dimensional gentle algebra of finite global dimension with surface model $(\Sigma,M, \zD^*)$ of genus $g$ with $b$ boundary components $\partial_1, \ldots, \partial_b$ where~$\partial_i$ has $m_i$ stops and winding number $\omega_i$. Then the Coxeter polynomial of $A$ is given by 
\begin{equation}\label{eq:coxepoly}
\psi(t) = (t-1)^{{-2+2g+b}} \prod_{i \in \{1, \ldots, b\} } (t^{m_i} - (-1)^{\omega_i}).
\end{equation}
\end{proposition}

\begin{proof}
In \cite[Theorem 6.17]{GM24} it is shown that $$\psi(t)=(t-1)^{|Q_1|-|Q_0|} \prod_{\substack{m>0 \\ n\geq 0}} (t^m - (-1)^{m+n})^{\varphi(m,n)}$$ where for any non-negative integers $m$ and $n$,  $\varphi(m,n)$ is the number of pairs of $(m,n)$ in the AG-invariant of $A$, defined in \cite{AG08}. It follows from \cite{APS23} that $|Q_1|-|Q_0| = -2+2g+b$ and in \cite{LP20}, see also \cite{OPS18}, it is shown that $\varphi(m,n)$ is given by the number of boundary components with $m$ stops and winding number $m-n$. The result then follows, noticing that $(-1)^{m+n}=(-1)^{m-n}$.
\end{proof}

\section{Examples and remarks}
\label{section_examples}
This section consists of a collection of examples of graded gentle algebras and marked surfaces. Using Theorems~\ref{thm:main1} and~\ref{thm:main2}, we compute the entropy of the Serre functor and the Serre dimension for the derived categories/Fukaya categories for these examples.  In tables~\ref{table-1} and~\ref{table-2}, we summarize the findings for these examples.

We briefly recall the following notations: the geometric model $(\Sigma,M,\Delta^*, G)$ of a graded gentle algebra $A$ is given by a  surface $\Sigma$ with marking $M$, admissible $\rpoint$-dissection~$\Delta^*$, and grading $G$ (see Definitions~\ref{definition:marked surface},~\ref{definition:addmissable dissections in prelimilary}, and~\ref{definition:graded marked surface}). The algorithm for constructing the geometric model of a (graded) gentle quiver is given in Definition~\ref{def:Sigma-from-A}.

We denote the connected components of the boundary $\partial$ by $\partial_1,\ldots,\partial_b$. We assume that~$\partial_i$ contains $m_i>0$ marked $\gpoint$-points  and there are $l_i$ arcs in $\Delta^*$ attached to $\partial_i$  (an arc is counted with multiplicity two if both ends are attached to $\partial_i$). We denote by $u_i:=l_i-m_i$ the number of polygons (counted with multiplicities) in the dissection $\Delta^*$ that contact $\partial_i$ at a corner. The winding number~$\omega_i$ around $\partial_i$ is defined by formula~\eqref{eq_w} and we define $n_i:=m_i-\omega_i$.  The collection of pairs $(m_i,n_i)$ for $i=1,\ldots,b$ is known as the AG-invariant of $A$ (or of $(\Sigma,M,\Delta^*, G)$).

We note that in the ungraded case $\omega_i=m_i-u_i$ by~\eqref{eq_w} and $n_i=u_i$. 

By Theorem~\ref{thm:main2prime}, the upper and lower Serre dimension of $\dba$ are given by
$$\uSdim \dba = \max \mathcal N \;\; \mbox{ and }\;\;  \lSdim \dba=\min \mathcal N, $$
where
$$\mathcal N=\left\{\frac{n_1}{m_1},\ldots,\frac{n_b}{m_b},0\right\}.$$
By Theorem~\ref{thm:main1prime} the entropy of the Serre functor on $\dba$ is given by 
$$H_t(\mathbb{S}) = \left\{\begin{array}{ll}
						(\uSdim\dba)\cdot t, & t\geq 0; \\ & \\
						(\lSdim\dba)\cdot t, & t\leq 0.
                        \end{array}\right.
$$	
Thus, we will only compute the upper and lower Serre dimensions. In Tables~\ref{table-1} and~\ref{table-2} we collect all the information for Examples~\ref{example_1} to~\ref{example_12}: gentle quivers, geometric models, ratios $n_i/m_i$, and dimensions $\lSdim, \uSdim$. The surfaces in Examples~\ref{example_5},~\ref{example_6}, and~\ref{example_9} are tori: we identify the opposite sides of the squares.

\begin{table}
\centering
\begin{tabular}{|c|c|c|c|c|c|}
\hline
& quiver & model & $n_i/m_i$ & $\lSdim$ & $\uSdim$ \\ \hline
1& $\bul_1 \to \bul_2\to \ldots \to \bul_n$ &
        \begin{tabular}{c}\begin{tikzpicture}[scale=0.72]
                \path[use as bounding box] (-2.3,-2.3) rectangle (2.3,2.3);
                \draw[thick] (0,0) circle (2);  
      		\foreach \u in {1,...,4} 
    		\draw[red!60, thick, bend left]  ({2*cos(72*\u-36)},{2*sin(72*\u-36)})       to ({2*cos(72*\u+36)},{2*sin(72*\u+36)});
                \foreach \u in {1,2,3,4} 
                \foreach \u in {0,...,4} 
  			\draw[fill=red, ] ({2*cos(72*\u+36)},{2*sin(72*\u+36)}) circle (2pt);
                \foreach \u in {0,...,4} 
			\draw[fill=white] ({2*cos(72*\u)},{2*sin(72*\u)}) circle (2pt);
            \draw[font=\scriptsize] (1.9,-1.9) node {$(n=4)$};
	  \end{tikzpicture}\end{tabular}
         & $\frac{n-1}{n+1}$ & $\frac{n-1}{n+1}$ & $\frac{n-1}{n+1}$ \\  \hline
2& $\xymatrix{\bul_1\ar@<1mm>[r]\ar@<-1mm>[r]&\bul_2}$ &
        \begin{tabular}{c} \begin{tikzpicture}[scale=0.72]
            \path[use as bounding box] (-2.3,-2.3) rectangle (2.3,2.3);            
            \draw[thick] (0,0) circle (2);  
            \draw[thick, fill=gray!40] (0,0) circle (0.5);  
            \draw[red!60, thick]  (0.5,0)  to[out=45, in=-45] (0.5,1) to[out=135, in=45]     (-1,1) to (-2,0);
            \draw[red!60, thick]  (0.5,0)  to[out=-45, in=45] (0.5,-1) to[out=-135, in=-45] (-1,-1) to (-2,0);
            \draw[fill=red] (0.5,0) circle (2pt);
            \draw[fill=red] (-2,0) circle (2pt);
            \draw[fill=white] (-0.5,0) circle (2pt);
            \draw[fill=white] (2,0) circle (2pt);
        \end{tikzpicture}\end{tabular}
         & $\frac 11, \frac 11$ & $      1$ & $1$\\ \hline
3& \begin{tabular}{c}  $\xymatrix{\bul_1\ar[r]^x&\bul_2\ar[r]^{y_1}\ar@/_1em/[r]_{y_2}&\bul_3}$ \\ $xy_2=0$            \end{tabular} & \begin{tabular}{c} \begin{tikzpicture}[scale=0.72]
            \path[use as bounding box] (-2.3,-2.3) rectangle (2.3,2.3);
             \draw[thick] (0,0) circle (2);  
             \draw[thick, fill=gray!40] (0,0) circle (0.5);  
             \draw[red!60, thick]  (0.5,0)  to[out=45, in=-45] (0.5,1) to[out=135, in=45]     (-1,1) to (-2,0);
             \draw[red!60, thick]  (0.5,0)  to[out=-45, in=45] (0.5,-1) to[out=-135, in=-45] (-1,-1) to (-2,0);
             \draw[red!60, thick]   (0,2)   to[out=-145, in=70]  (-2,0);
            \draw[fill=red] (0.5,0) circle (2pt);
            \draw[fill=red] (-2,0) circle (2pt);
            \draw[fill=red] (0,2) circle (2pt);
            \draw[fill=white] (-0.5,0) circle (2pt);
            \draw[fill=white] (2,0) circle (2pt);
            \draw[fill=white] ({2*cos(135)},{2*sin(135)}) circle (2pt);
        \end{tikzpicture} \end{tabular} & $\frac 11, \frac 22$ & $1$ & $1$\\ \hline
4& \begin{tabular}{c} $\xymatrix{ & \bul_2\ar[rd]^y & \\ \bul_1\ar[rr]\ar[ru]^x && \bul_3}$ \\ $xy=0$                  \end{tabular} & 
        \begin{tabular}{c} \begin{tikzpicture}[scale=0.72]
            \path[use as bounding box] (-2.3,-2.3) rectangle (2.3,2.3);
             \draw[thick] (0,0) circle (2);  
             \draw[thick, fill=gray!40] (0,0) circle (0.5);  
             \draw[red!60, thick]  (0.5,0)  to[out=45, in=-45] (0.5,1) to[out=135, in=45]     (-1,1) to (-2,0);
             \draw[red!60, thick]  (0.5,0)  to[out=-45, in=45] (0.5,-1) to[out=-135, in=-45] (-1,-1) to (-2,0);
             \draw[red!60, thick]   (0.5,0)   to   (2,0);
            \draw[fill=red] (0.5,0) circle (2pt);
            \draw[fill=red] (-2,0) circle (2pt);
            \draw[fill=red] (2,0) circle (2pt);
            \draw[fill=white] (-0.5,0) circle (2pt);
            \draw[fill=white] ({2*cos(60)},{2*sin(60)}) circle (2pt);
            \draw[fill=white] ({2*cos(-60)},{2*sin(-60)}) circle (2pt);
        \end{tikzpicture} \end{tabular} & $\frac 12, \frac 21$ & $1/2$ & $2$\\ \hline
5& \begin{tabular}{c}
              $\xymatrix{\bul \ar@/^1em/[r]^{x_1}\ar@/_1em/[r]_{y_1} & \bul \ar@/^1em/[r]^{x_2}\ar@/_1em/_{y_2}[r] & \bul    }$  \\
              $x_1x_2=y_1y_2=0$ 
         \end{tabular} & 
         \begin{tabular}{c} \begin{tikzpicture}[scale=0.72]\path[use as bounding box] (-2.3,-2.3) rectangle (2.3,2.3);
             \draw[thick, dotted] (-2,-2) rectangle  (2,2);  
             \draw[thick, fill=gray!40] (0,0) circle (0.5);  
             \draw[red!60, thick]  ({0.5*cos(45)},{0.5*sin(45)})  to (2,2);
             \draw[red!60, thick]  ({0.5*cos(45)},{0.5*sin(45)})  to (2,0);
             \draw[red!60, thick]  ({0.5*cos(45)},{0.5*sin(45)})  to (0,2);
             \draw[red!60, thick]  ({0.5*cos(-135)},{0.5*sin(-135)})  to (-2,-2);
             \draw[red!60, thick]  ({0.5*cos(-135)},{0.5*sin(-135)})  to (-2,0);
             \draw[red!60, thick]  ({0.5*cos(-135)},{0.5*sin(-135)})  to (0,-2);
             \draw[fill=red] ({0.5*cos(45)},{0.5*sin(45)}) circle (2pt);
             \draw[fill=red] ({0.5*cos(-135)},{0.5*sin(-135)}) circle (2pt);
             \draw[fill=white] ({0.5*cos(-45)},{0.5*sin(-45)}) circle (2pt);
             \draw[fill=white] ({0.5*cos(135)},{0.5*sin(135)}) circle (2pt);
        \end{tikzpicture} \end{tabular} & $\frac 42$ & $1$ & $2$ \\   \hline
6&  \begin{tabular}{c}
              $\xymatrix{\bul \ar@/^1em/[r]^{x_1}\ar@/_1em/[r]_{y_1} & \bul \ar@/^1em/[r]^{x_2}\ar@/_1em/_{y_2}[r] & \bul \ar@/^1em/[r]^{x_3}\ar@/_1em/_{y_3}[r] & \bul    }$  \\
              $x_{1}x_2=x_2x_3=0$ \\
              $y_{1}y_2=y_2y_3=0$
         \end{tabular} &  
         \begin{tabular}{c} \begin{tikzpicture}[scale=0.72]\path[use as bounding box] (-2.3,-2.3) rectangle (2.3,2.3);
             \draw[thick, dotted] (-2,-2) rectangle  (2,2);  
             \draw[thick, fill=gray!40] (1,1) circle (0.5);  
             \draw[thick, fill=gray!40] (-1,-1) circle (0.5);  
             
             \draw[red!60, thick]  ({1+0.5*cos(-135)},{1+0.5*sin(-135)})  to[out=180, in=-90] (0,2);
             \draw[red!60, thick]  ({1+0.5*cos(-135)},{1+0.5*sin(-135)})  to[out=210, in=-45] (-2,2);
             \draw[red!60, thick]  ({1+0.5*cos(-135)},{1+0.5*sin(-135)})  to[out=240, in=30] (0,0);
             \draw[red!60, thick]  ({1+0.5*cos(-135)},{1+0.5*sin(-135)})  to[out=270, in=180] (2,0);
   
             \draw[red!60, thick]  ({-1+0.5*cos(45)},{-1+0.5*sin(45)})   to[out=0, in=90] (0,-2);
             \draw[red!60, thick]  ({-1+0.5*cos(45)},{-1+0.5*sin(45)})   to[out=30, in=135] (2,-2);
             \draw[red!60, thick]  ({-1+0.5*cos(45)},{-1+0.5*sin(45)})   to[out=60, in=210] (0,0);
             \draw[red!60, thick]  ({-1+0.5*cos(45)},{-1+0.5*sin(45)})   to[out=90, in=0] (-2,0);
         
             \draw[fill=red] ({1+0.5*cos(-135)},{1+0.5*sin(-135)}) circle (2pt);
             \draw[fill=red] ({-1+0.5*cos(45)},{-1+0.5*sin(45)}) circle (2pt);
             \draw[fill=white] ({1+0.5*cos(45)},{1+0.5*sin(45)}) circle (2pt);
             \draw[fill=white] ({-1+0.5*cos(-135)},{-1+0.5*sin(-135)}) circle (2pt);
        \end{tikzpicture} \end{tabular}& $\frac 31, \frac 31$ & $1$ & $3$ \\   \hline
\end{tabular}
\caption{Examples, part 1}\label{table-1}
\end{table}

\begin{table}
\centering
\begin{tabular}{|c|c|c|c|c|c|}
\hline
& quiver & model & $n_i/m_i$ & $\lSdim$ & $\uSdim$ \\ \hline
7& \begin{tabular}{c}
              $\xymatrix{\bul  \ar@(ur,u)[rr]^x  && \bul  \ar@(d,dr)[ll]^y }$  \\
              $yx=0$ 
         \end{tabular} & 
         \begin{tabular}{c} \begin{tikzpicture}[scale=0.72]\path[use as bounding box] (-2.3,-2.3) rectangle (2.3,2.3);
             \draw[thick] (0,0) circle (2);  
             \draw[thick, fill=gray!40] (0,0) circle (0.5);  
             \draw[red!60, thick]  (0.5,0)  to[out=45, in=-45] (0.5,1) to[out=135, in=90]     (-1.25,0);
             \draw[red!60, thick]  (0.5,0)  to[out=-45, in=45] (0.5,-1) to[out=-135, in=-90]     (-1.25,0);
             \draw[red!60, thick]   (0.5,0)   to   (2,0);
            \draw[fill=red] (0.5,0) circle (2pt);
            \draw[fill=red] (2,0) circle (2pt);
            \draw[fill=white] (-0.5,0) circle (2pt);
            \draw[fill=white] (-2,0) circle (2pt);
        \end{tikzpicture} \end{tabular} & $\frac 01, \frac 21$ & $0$ & $2$ \\   \hline
8&  \begin{tabular}{c} $\xymatrix{ & \bul_2\ar@/^1em/[rd]^y & \\ \bul_1\ar@/^1em/[ru]^x && \bul_3\ar[ll]^z}$ \\             $zx=yz=0$ \end{tabular}  & 
         \begin{tabular}{c} \begin{tikzpicture}[scale=0.72]\path[use as bounding box] (-2.3,-2.3) rectangle (2.3,2.3);
            \draw[thick] (0,0) circle (2);  
            \draw[thick,  fill=gray!40] (0,0) circle (0.5);  
            \draw[red!60, thick]  (0,0.5)  to[out=135, in=30] (-2,0);
            \draw[red!60, thick]  (0,-0.5)  to[out=-135, in=-30] (-2,0);
            \draw[red!60, thick]  (1.25,0)  to[out=90, in=0] (0,1.25) to[out=180, in=60] (-2,0);
            \draw[red!60, thick]  (1.25,0)  to[out=-90,in=0] (0,-1.25) to[out=180, in=-60] (-2,0);   
            \draw[fill=red] (0, 0.5) circle (2pt);
            \draw[fill=red] (0, -0.5) circle (2pt);
            \draw[fill=red] (-2,0) circle (2pt);
            \draw[fill=white] (-0.5,0) circle (2pt);
            \draw[fill=white] (0.5,0) circle (2pt);
            \draw[fill=white] (2,0) circle (2pt);
        \end{tikzpicture}       \end{tabular} & $\frac 02, \frac 31$ & $0$ & $3$\\ \hline
9&  \begin{tabular}{c} \resizebox{5cm}{!}{
         $\begin{tikzpicture}[scale=1, >=Latex]
            \node[circle, fill=black, inner sep=2pt, label=135:1] (A) at (-2,0) {};
            \node[circle, fill=black, inner sep=2pt, label=-45:2] (B) at (2,0) {};
            \draw[-{Stealth[scale=1.2]}, thick, line width=1pt]  (A) .. controls (0,1.5) and (4,0) .. 
            node[pos=0.28, above] {$x$} (B);
            \draw[-{Stealth[scale=1.2]}, thick, >=Latex, line width=1pt] (A) .. controls (-4,0) and (0,-1.5) ..   node[pos=0.72, above] {$z$} (B);
            \draw[-{Stealth[scale=1.2]}, thick, >=Latex, line width=1pt] (B) to  node[midway, above] {$y$} (A);
         \end{tikzpicture}$} \\ $yx=zy=0$ \end{tabular}  & 
         \begin{tabular}{c} \begin{tikzpicture}[scale=0.72]\path[use as bounding box] (-2.3,-2.3) rectangle (2.3,2.3);
             \draw[thick, dotted] (-2,-2) rectangle  (2,2);  
             \draw[thick, fill=gray!40] (0,0) circle (0.5);  
             \draw[red!60, thick]  ({0.5*cos(45)},{0.5*sin(45)})  to[out=0, in=90] (1,-2);
             \draw[red!60, thick]  ({0.5*cos(45)},{0.5*sin(45)})  to[out=30, in=180] (2,1);
             \draw[red!60, thick]  ({0.5*cos(45)},{0.5*sin(45)})  to[out=60, in=-90] (1,2);
             \draw[red!60, thick]  ({0.5*cos(45)},{0.5*sin(45)})  to[out=90, in=0] (-2,1);
             \draw[fill=red] ({0.5*cos(45)},{0.5*sin(45)}) circle (2pt);
             \draw[fill=white] ({0.5*cos(-135)},{0.5*sin(-135)}) circle (2pt);
        \end{tikzpicture} \end{tabular}&  $\frac 31$ & $1$ & $3$\\ \hline
10& \begin{tabular}{c} $\xymatrix{ \bul\ar[rd]^{x_3} &
        & \bul\ar[ld]_{y_3}\\ & \bul\ar[ld]^{x_1}\ar[rd]_{y_1} & \\           \bul\ar[uu]^{x_2} && \bul \ar[uu]_{y_2}}$ \\ $x_{i}x_j=y_iy_j=0$ \end{tabular} & 
        \begin{tabular}{c} \begin{tikzpicture}[scale=0.72] \path[use as bounding box] (-2.3,-2.3) rectangle (2.3,2.3);
                \draw[thick] (0,0) circle (2);  
			    \draw[red!60, thick]  ({2*cos(45)},{2*sin(45)}) to (0.75,0); 
                \draw[red!60, thick]  ({2*cos(-45)},{2*sin(-45)}) to (0.75,0);
                \draw[red!60, thick]  ({2*cos(135)},{2*sin(135)}) to (-0.75,0);
                \draw[red!60, thick]  ({2*cos(-135)},{2*sin(-135)}) to (-0.75,0);
                \draw[red!60, thick]  (0.75,0) to (-0.75,0);
                \foreach \u in {0,...,3} 
  				\draw[fill=red, ] ({2*cos(90*\u+45)},{2*sin(90*\u+45)}) circle (2pt);
                \draw[fill=red] (0.75,0) circle (2pt); 
                \draw[fill=red] (-0.75,0) circle (2pt); 
                \foreach \u in {0,...,3} 
  			  \draw[fill=white] ({2*cos(90*\u)},{2*sin(90*\u)}) circle (2pt);
        \end{tikzpicture}   \end{tabular} & $\frac 04$ & $0$ & $1$\\ \hline
11& \begin{tabular}{c} $\xymatrix{ \bul_1\ar[r]^{x_1} & \bul_2\ar[d]^{x_2}\\ \bul_n\ar[u]^{x_n} &                          \bul_3\ar@{.}[l]}$ \\ $x_{i}x_{i+1}=x_nx_1=0$ \end{tabular} &
        \begin{tabular}{c} \begin{tikzpicture}[scale=0.72] \path[use as bounding box] (-2.3,-2.3) rectangle (2.3,2.3);
                \draw[thick] (0,0) circle (2);  
			    \foreach \u in {0,...,4} 
    				\draw[red!60, thick]  ({2*cos(72*\u-36)},{2*sin(72*\u-36)}) to (0,0);
                \foreach \u in {0,...,4} 
  				\draw[fill=red, ] ({2*cos(72*\u+36)},{2*sin(72*\u+36)}) circle (2pt);
                \draw[fill=red] (0,0) circle (2pt); 
                \foreach \u in {0,...,4} 
  				\draw[fill=white] ({2*cos(72*\u)},{2*sin(72*\u)}) circle (2pt);
                \draw[font=\scriptsize] (1.9,-1.9) node {$(n=5)$};
        \end{tikzpicture}  \end{tabular} & $\frac 0n$ & $0$ & $0$\\ \hline
12 & $A(r,n,m), 0<r<n, m\ge 0$ & see Figure \ref{fig:12a}  &  $\frac{m}{m+r}, \frac{n}{n-r}$ &  $\frac{m}{m+r} $ & $\frac{n}{n-r}$\\
          & $A(n,n,m), n>0, m\ge 0$ & see Figure \ref{fig:12b} & $\frac{m}{m+r}$ &  $\frac{m}{m+r} $ & $\frac{m}{m+r}$\\ \hline
\end{tabular}
\caption{Examples, part 2}\label{table-2}
\end{table}

\begin{example}
\label{example_1}
Let $Q=A_n$ be the Dynkin quiver of type $A_n$ with  linear orientation, see Table~\ref{table-1}, line 1. Let $\calt_{n+1}:=\Db(kA_n)$. The geometric model of $kA_n$ is a disc with $n+1$ $\gpoint$-points on the boundary and the $\Delta^*$-dissection as shown in Table~\ref{table-1}.  It was established by~\cite{AssemHappel}  that a basic algebra is derived equivalent to $kA_n$ if and only if it is a gentle algebra whose underlying graph is a tree with $n$ vertices. Hence, all gentle trees with $n$ vertices share the same derived category $\calt_{n+1}$, their geometric models are given by a disc with $n+1$ $\gpoint$-points on the boundary but have different dissections. 
Moreover, any Fukaya category $\mathcal W(\Sigma, M, \eta)$, where $\Sigma$ is a disc without stops in the interior, is equivalent to $\calt_n$, where  $n$ is the number of stops. Indeed, all line fields on a disc are homotopic, so the only parameter is the cardinality of $M_{\gpoint}$.  In particular, the algebra~$kA_n$ with a grading produces a  category equivalent to $\calt_{n+1}$.
Category $\calt_{n+1}$ is fractionally Calabi--Yau: one has an isomorphism of functors $\SSS^{n+1}\cong [n-1]$.
\end{example}

\begin{example}
\label{example_2}
Let $K_2$ be the Kronecker quiver, see Table~\ref{table-1}, line 2. Its derived category of finitely generated modules $\Db(kK_2)$ is equivalent to the derived category of coherent sheaves on $\mathbb P^1$. The Serre functor on the latter is given by the tensoring with the canonical bundle $\omega_{\mathbb P^1}$ followed by the shift by $1$, which gives a geometric explanation of the equalities $\lSdim=\uSdim=1$.   
\end{example}

\begin{example}
\label{example_3}
Let $A$ be the gentle algebra as in Table~\ref{table-1}, line 3. The derived category of modules $\Db(A)$ is equivalent to the derived category of coherent sheaves on the weighted projective line $\mathbb P^1_{(\mathbf 2)}$ with one weighted point of multiplicity $2$ (see~\cite{GL}). As in Example~\ref{example_2}, this gives a geometric interpretation of the equalities $\lSdim=\uSdim=1$.   
\end{example}

\begin{example}
\label{example_4}
Let $A$ be the gentle algebra as in Table~\ref{table-1}, line 4. The Serre dimension of $\Db(A)$ has been  computed previously in ~\cite[Example 9.1]{E22} by  directly working with complexes of $A$-modules.
\end{example}

\begin{example}
\label{example_5}
Let $A$ be the gentle algebra as in Table~\ref{table-1}, line 5 (sometimes referred to as ``Bondal quiver'' or ``nodal quiver''). The geometric model of $A$ is a torus with one hole, the only boundary component of $\Sigma$ has two $\gpoint$-points and six $\rpoint$-arcs in $\Delta^*$ attached. One has $\mathcal N=\{n_1/m_1,1\}=\{4/2,1\}$. In this case, $\min\mathcal N$ equals not the minimal value of $n_i/m_i$ (which is $2$) but $1$. This is an example that shows  why including $1$ in the definition of   $\mathcal N$ is necessary for the formula giving $\lSdim$. 
\end{example}

\begin{example}
\label{example_6}
Let $A$ be the gentle algebra as in Table~\ref{table-1}, line 6. This example is similar to Example~\ref{example_5}: 
the geometric model of $A$ is a torus with two holes. As before one has $\mathcal N=\{n_1/m_1,1\}=\{4/2,1\}$, and  $\min\mathcal  N=1$. This example appears in \cite{ChangHaidenSchroll}, see also \cite{NakagoTakahashi},  where it is proven that the action of the braid group on the set of full exceptional collections in $\Db( A)$ by mutations is not transitive.
\end{example}

\begin{example}
\label{example_7}
Let $A$ be the Auslander algebra of the ring of dual numbers, see Table~\ref{table-2}, line 7. The geometric model of $A$ is an annulus, and $\mathcal N=\{n_1/m_1,n_2/m_2,1\}=\{0/1,2/1,1\}$. The category $\Db( A)$ has an exceptional collection whose graded endomorphism algebra is the Kronecker algebra from Example~\ref{example_2} with non-trivial grading: the degrees of two arrows are $0$ and $1$. Hence, Serre dimension of $\Db( A)$ can also be computed using \cite[Prop. 8.4]{E22}.
\end{example}

\begin{example}
\label{example_8}
Let $A$ be the gentle algebra as in Table~\ref{table-2}, line 8. The Serre dimension of $\Db(A)$ has been computed previously in ~\cite[Example 9.4]{E22} by directly working with complexes of $A$-modules.
\end{example}

\begin{example}
\label{example_9}
Let $A$ be the gentle algebra as in Table~\ref{table-2}, line 9. Its geometric model is a torus with a hole and with one marked $\gpoint$-point. 
The only ratio 
$n_1/m_1$ is $3/1$, however the lower Serre dimension is $1$.
\end{example}

The next two examples are of infinite global dimension.

\begin{example}
\label{example_10}
Let $A$ be the gentle algebra as in Table~\ref{table-2}, line 10.
The geometric model of  $A$ is a disc with two $\rpoint$-points in the interior. Here $\mathcal N=\{n_1/m_1,1\}=\{0/4,1\}$. 
Hence, the maximum of $\mathcal N$ is not the maximal value of $n_i/m_i$ (which is $0$) but $1$. This is an example that shows why including 1 in the
definition of $\mathcal N$ is necessary for the formula giving  $\uSdim$.  
\end{example}

\begin{example}
\label{example_11}
Let $A=k\tilde A_{n-1}/R^2$ be the path algebra of a cycle of length $n$ modulo all quadratic relations (see Table~\ref{table-2}, line 11). This is a radical square zero cyclic  Nakayama algebra. Consider the grading on $A$ where  the arrow from $i$ to $i+1$  is in degree $d_i$, for $1 \leq i \leq n$ and where we set $n+1 = 1$. 
The geometric model of $A$ is a disc with one $\rpoint$-point in the interior and $n$ $\gpoint$-points  on the boundary, see the picture in the table. 
Set $d=\sum_i d_i$.  It follows from~\eqref{eq_w} that the winding number $\omega$ around the only boundary is 
$$\omega_1=\sum_{i=1}^n (1-d_i)=n-d$$
and $n_1=m_1-\omega_1=d$. Our results do not apply to these algebras directly. However, below we compute the Serre dimensions  explicitly.

Let $\ell_i$ be the $\gpoint$-arc corresponding to vertex $i$, then 
$\tau\ell_i\sim \ell_{i+1}$ and (by Definition~\ref{def_degree-of-orint} and the definition of $\tau f$) $\tau(\ell_i,f)\sim (\ell_{i+1}, f+1-d_i)$ for any grading $f$ of $\ell_i$. Therefore $\tau^n (\ell_i,f)\sim (\ell_i, f+n-d)$ for all $i$. It follows that $\tau^n=[d-n]$ and hence $\SSS^n=[d]$ (on arcs from $\Delta$). We obtain that 
$$\uSdim \dba =\lSdim \dba=d/n=n_1/m_1.$$
In the ungraded case this gives $\uSdim \dba =\lSdim \dba=0$.
One can see that the formulas from Theorem~\ref{thm:main2prime} do not  give  the correct  answer.

Furthermore, note that a graded marked surface $(\Sigma, M, \eta)$, where $\Sigma$ is a disc with one stop in the interior, is 
determined by two parameters: the number $n$ of stops on the boundary and the winding number $\omega$ of $\eta$ around the boundary.
Let $\calt_{n,\omega}$ be the corresponding Fukaya category, then $\dba\cong \calt_{n,n-d}$.
Note that categories $\dba$ from this example provide all possible values of the parameters $n$ and $\omega$. Hence, any graded gentle algebra, whose model is a disc with one stop in the interior, is derived equivalent to some radical square zero graded cyclic Nakayama algebra.
\end{example}

In the following example, we consider derived discrete algebras and their 
perfect derived categories, see Table~\ref{table-2}, line 12. Entropy of the Serre functor for these categories has also been calculated in~\cite{C25}.
\begin{example}
\label{example_12} 
Let $Q=Q(n,m)$ be the quiver
$$
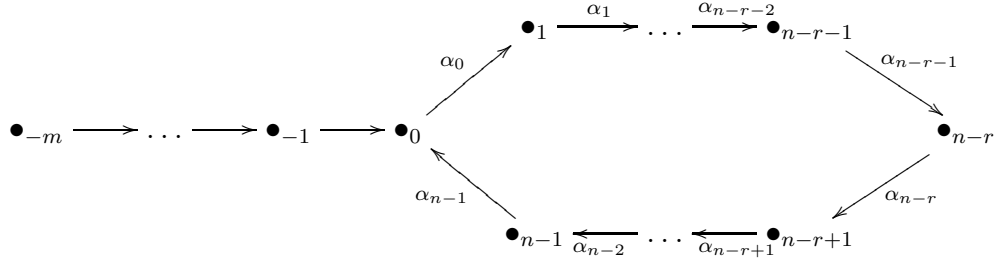
\xymatrix{
&&&& \bul_1 \ar[r]^{\alpha_1} & \ldots \ar[r]^{\alpha_{n-r-2}} & \bul_{n-r-1}\ar[rd]^{\alpha_{n-r-1}} & \\
\bul_{-m}\ar[r] & \ldots\ar[r] & \bul_{-1}\ar[r] & \bul_0 \ar[ru]^{\alpha_0} &&&& \bul_{n-r}\ar[ld]^{\alpha_{n-r}}\\
&&&& \bul_{n-1}\ar[lu]^{\alpha_{n-1}} & \ldots\ar[l]^{\alpha_{n-2}} & \bul_{n-r+1}\ar[l]^{\alpha_{n-r+1}}
}$$
where $1\le n$, $m\ge 0$. For $1\le r\le n$, let $A(r,n,m)$ be the path algebra of $Q(n,m)$ over~$k$ with $r$ relations
$$\alpha_{n-1}\alpha_{0}=\alpha_{n-2}\alpha_{n-1}=\ldots=\alpha_{n-r}\alpha_{n-r+1}=0.$$
Clearly, $A(r,n,m)$ is finite-dimensional and gentle. The global dimension of $A(r,n,m)$ is finite if and only if $r<n$.

The derived categories of these algebras are known \cite{Vossieck}
to be \emph{discrete}: roughly speaking, indecomposable objects in these categories have no moduli. Moreover, by \cite{BobinskiGeissSkowronski}
if the  derived category $\Db( A)$ of some algebra is indecomposable and discrete then $\Db( A)\cong \Db( k\Gamma)$, where $\Gamma$ is a Dynkin quiver, or $\Db( A)\cong \Db(A(r,n,m))$
for some $r,n,m$. Thus, algebras $A(r,n,m)$ are central in the study of discrete derived categories. 

First, assume $r<n$ (finite global dimension case). Then the geometric model of $A(r,n,m)$ is an annulus,  see Figure~\ref{fig:12a}.
Here, one identifies the left and right sides of the rectangle to get an annulus, and the arcs in $\Delta^*$ are labeled with the corresponding vertices of $Q_{n,m}$. One has 
\begin{align*}
m_{\mathrm{bottom}}&=m+r, & u_{\mathrm{bottom}}=n_{\mathrm{bottom}}=&m, & \omega_{\mathrm{bottom}}=&r,\\ 
m_{\mathrm{top}}&=n-r, & u_{\mathrm{top}}=n_{\mathrm{top}}=&n, &\omega_{\mathrm{top}}=&-r.
\end{align*}
Hence by Theorem~\ref{thm:main2prime} one has
$$\lSdim(\Db(A(r,n,m)))=\frac{m}{r+m}<1, \quad \uSdim(\Db(A(r,n,m)))=\frac{n}{n-r}>1.$$

\begin{figure}
    \centering
    \begin{tikzpicture}[scale=0.6] 
                \draw[thick] (-8,-4) to  (8,-4);  
                \draw[thick] (-8,4) to  (8,4);  
                \draw[red!60, thick] (-8,-4) to  (-8,4);  
                \draw[red!60, thick] (8,-4) to  (8,4);  
			    \foreach \u in {1,3,4} 
    				\draw[red!60, thick]  ({2*\u},-4) arc[start angle = 0, end angle = 180, radius = 1];
                \foreach \u in {2} 
    				\draw[red!60, thick, dashed]  ({2*\u},-4) arc[start angle = 0, end angle = 180, radius = 1];
                \foreach \u in {-1,1,2} 
    				\draw[red!60, thick]  (4*\u,4) arc[start angle = -45, end angle = -135, radius = {4*cos(45)}];
                \foreach \u in {0} 
    				\draw[red!60, thick, dashed]  (4*\u,4) arc[start angle = -45, end angle = -135, radius = {4*cos(45)}];
                \draw[red!60, thick]  (-6,-4) -- (-8,4) node[midway, sloped, fill=white, inner sep=2pt, font=\scriptsize]{$n-1$};
                \draw[red!60, thick, dashed]  (-4,-4) to (-8,4);
                \draw[red!60, thick]  (-2,-4) -- (-8,4) node[midway, sloped, fill=white, inner sep=2pt, font=\scriptsize]{$n-r+1$};
                \foreach \u in {-4,...,4} 
                	\draw[fill=red] ({2*\u},-4) circle (2pt);
                \foreach \u in {-2,...,2} 
  				\draw[fill=red] ({4*\u},4) circle (2pt);
                \foreach \u in {-3,...,4} 
  				\draw[fill=white] ({2*\u-1},-4) circle (2pt);
                \foreach \u in {-1,...,2} 
  				\draw[fill=white] ({4*\u-2},4) circle (2pt);
                \draw[font=\scriptsize, red!60] (7,-3) node[fill=white, inner sep=2pt] {$-1$}; 
                \draw[font=\scriptsize, red!60] (5,-3) node[fill=white, inner sep=2pt] {$-2$};
                \draw[font=\scriptsize, red!60] (1,-3) node[fill=white, inner sep=2pt] {$-m$};
                \draw[font=\scriptsize, red!60] (6,3.2) node[fill=white, inner sep=2pt] {$1$}; 
                \draw[font=\scriptsize, red!60] (2,3.2) node[fill=white, inner sep=2pt] {$2$}; 
                \draw[font=\scriptsize, red!60] (-6,3.2) node[fill=white, inner sep=2pt] {$n-r$};
                \draw[font=\scriptsize, red!60] (8,0) node[fill=white, inner sep=2pt] {$0$}; 
                \draw[font=\scriptsize, red!60] (-8,0) node[fill=white, inner sep=2pt] {$0$}; 
    \end{tikzpicture}
    \caption{Geometric model of the derived discrete algebra $A_{r,n,m}$, finite global dimension case $r<n$.  } 
    \label{fig:12a}
\end{figure}
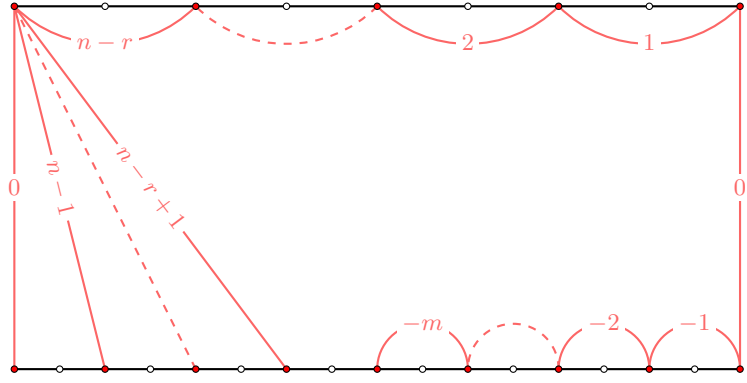


If $r=n$ (infinite global dimension case) the geometric model of $A(r,n,m)$ is
a disc with one stop in the interior, see Figure~\ref{fig:12b}.
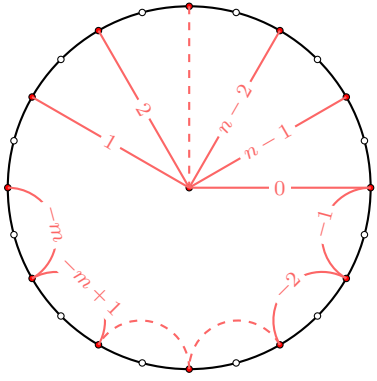
\begin{figure}
    \centering
    \begin{tikzpicture}[scale=0.6] 
                \draw[thick] (0,0) circle (4);  
                \foreach \u in {0,...,11} 
                	\draw[fill=red] ({4*cos(\u*30},{4*sin(\u*30)}) circle (2pt);
                \draw[fill=red] (0,0) circle (2pt);
                \foreach \u in {0,...,11} 
                	\draw[fill=white] ({4*cos(\u*30+15},{4*sin(\u*30+15)}) circle (2pt);
			    \draw[red!60, thick]  ({4*cos(0*30},{4*sin(0*30)}) -- (0,0)  node[midway, sloped,fill=white, inner sep=2pt,font=\tiny] {$0$}; 
                \draw[red!60, thick]  ({4*cos(1*30},{4*sin(1*30)}) -- (0,0)  node[midway, sloped,fill=white, inner sep=2pt,font=\tiny] {$n-1$};
                \draw[red!60, thick]  ({4*cos(2*30},{4*sin(2*30)}) -- (0,0)  node[midway, sloped,fill=white, inner sep=2pt,font=\tiny] {$n-2$}; 
                \draw[red!60, thick, dashed]  ({4*cos(3*30},{4*sin(3*30)}) -- (0,0)  ;
                \draw[red!60, thick]  ({4*cos(4*30},{4*sin(4*30)}) -- (0,0)  node[midway, sloped,fill=white, inner sep=2pt,font=\tiny] {$2$}; 
                \draw[red!60, thick]  ({4*cos(5*30},{4*sin(5*30)}) -- (0,0)  node[midway, sloped,fill=white, inner sep=2pt,font=\tiny] {$1$}; 
                \foreach \u in {0} 
    				\draw[red!60, thick]  ({4*cos(-\u*30},{4*sin(-\u*30)}) arc[start angle = 90-\u*30, end angle = 240-\u*30, radius = {4*tan(15)}] node[midway, sloped,fill=white, inner sep=2pt,font=\tiny] {$-1$}; 
                \foreach \u in {1} 
    				\draw[red!60, thick]  ({4*cos(-\u*30},{4*sin(-\u*30)}) arc[start angle = 90-\u*30, end angle = 240-\u*30, radius = {4*tan(15)}] node[midway, sloped,fill=white, inner sep=2pt,font=\tiny] {$-2$}; 
                \foreach \u in {2,3} 
    				\draw[red!60, thick, dashed]  ({4*cos(-\u*30},{4*sin(-\u*30)}) arc[start angle = 90-\u*30, end angle = 240-\u*30, radius = {4*tan(15)}];    
                \foreach \u in {4} 
    				\draw[red!60, thick]  ({4*cos(-\u*30},{4*sin(-\u*30)}) arc[start angle = 90-\u*30, end angle = 240-\u*30, radius = {4*tan(15)}] node[midway, sloped,fill=white, inner sep=2pt,font=\tiny] {$-m+1$}; 
                \foreach \u in {5} 
    				\draw[red!60, thick]  ({4*cos(-\u*30},{4*sin(-\u*30)}) arc[start angle = 90-\u*30, end angle = 240-\u*30, radius = {4*tan(15)}] node[midway, sloped,fill=white, inner sep=2pt,font=\tiny] {$-m$}; 
    \end{tikzpicture}
    \caption{Geometric model of the derived discrete algebra $A_{r,n,m}$, infinite global dimension case $r=n$}
    \label{fig:12b}
\end{figure}
We have 
$$m_1=m+r, u_1=n_1=m, \omega_1=r$$ in this case. Theorem~\ref{thm:main2prime} does not apply, but such models are treated in Example~\ref{example_11} and hence 
$$\lSdim(\Db(A(n,n,m)))=\uSdim(\Db(A(n,n,m)))=\frac{m}{n+m}<1.$$
In this case  category $\Db( A(n,n,m))$ is fractionally Calabi--Yau: $\SSS^{n+m}\cong [m]$.
\end{example}

\begin{example}
    \label{example_onehole}
Let $(\Sigma,M)$ be a marked surface of genus $g\ge 1$ such that  $\partial\cong S^1$ and $\eta$ be a line field on $\Sigma$. Let $\mathcal W$ be the corresponding Fukaya category and $m=|M_{\gpoint}|$. Then
\begin{equation*}
h_t({\SSS_{\mathcal W}})=
\begin{cases}
(1+\frac{4g-2}m)\cdot t, & t\ge 0,\\    
t, & t\le 0,
\end{cases}
\end{equation*}
$$
\uSdim(\mathcal W)=1+\frac{4g-2}m,\quad
\lSdim(\mathcal W)=1.
$$ 
In particular, entropy and dimension are independent of the line field. 

Indeed, the winding number of $\eta$ around the only boundary component of $\Sigma$ is $2-4g$ by Poincar\'e--Hopf index theorem (e.g. see \cite[(1.3)]{LP20}), and Theorems~\ref{thm:main1} and~\ref{thm:main2} provide the answer.

Such surfaces arise in Example~\ref{example_5} (with $g=1$, $m=2$) and Example~\ref{example_9} (with $g=1$, $m=1$). In Example~\ref{example_5}, the line field $\eta_A$ on the torus is homotopic to the constant line field, while in Example~\ref{example_9} it is not.  
\end{example}

\begin{example}
Let $A$ be the gentle algebra from Example~\ref{example_5}, and $(\Sigma,M, \eta)$ be the corresponding marked surface with a line field. Let $P$ be the $A$-module with dimension vector $(1,1,1)$ such that $x_1,x_2$ act by $1$ and $y_1,y_2$ act by $0$, it is an exceptional object in $\Db(A)$. The corresponding arc on $(\Sigma,M)$  connects two marked points $p_1,p_2\in M$ and goes parallel to $\partial$ in the negative direction. The orthogonal category $\calw:=P^{\perp}\subset \Db(A)$ is known (see~\cite{HaidenWu})  to be equivalent to the Fukaya category associated with the same model and $p_2$ removed: the torus with one hole, one marked point, and constant line field. Hence, $\calw$ is equivalent to the perfect derived category of the algebra from Example~\ref{example_9} with a grading (for example, one can take $\deg x=\deg z=0, \deg y=1$). Category~$\calw$ has no exceptional objects; this fact was used to demonstrate that Jordan--H\"older property for triangulated categories fails, see~\cite{K13}. It is also interesting to note that~$\calw$ has no stability conditions, see~\cite[Th. 1.1]{HaidenWu}. Finally, category $\calw$ was studied in~\cite{Sung}: in particular, the Serre functor on $\calw$ was related to the spherical twist by a $3$-spherical object in $\calw$ that corresponds to the $\gpoint$-loop around $\partial$ with endpoints at~$p_1$.

As a special case of Example~\ref{example_onehole} we get 
$$\lSdim(\calw)=1,\quad \uSdim(\calw)=3.$$
\end{example}

\begin{remark}
Note that the algebras in Examples~\ref{example_7}, \ref{example_8}, \ref{example_11} 
are $A(1,2,0)$, $A(2,3,0)$, $A(n,n,0)$ respectively. The algebra in Example~\ref{example_4} is derived Morita-equivalent, but not isomorphic, to $A(1,2,1)$.
\end{remark}

\begin{remark}
Note that any pair of rational numbers $\alpha,\beta$ such that $0\le \alpha<1<\beta$ can be realised as 
$(\alpha,\beta)=(\lSdim(\Db(A)), \uSdim(\Db(A))$ for some derived discrete algebra $A=A(r,n,m)$ of finite global dimension from Example~\ref{example_12}. Indeed, let 
$$\alpha=p/q, \beta=t/s, \quad\text{where}\quad p,q,t,s\in\Z, q>p\ge 0, t>s>0.$$
Take
$$r=(t-s)(q-p), n=t(q-p), m=p(t-s),$$
then by Theorem~\ref{thm:main2prime} one has
$\alpha=m/(r+m), \beta=n/(n-r)$ as required.
\end{remark}

\end{document}